\pdfoutput=1

\documentclass[english,11pt]{article}
\include{macros}

\begin{document}

\title{A Particle Algorithm for Mean-Field Variational Inference}\blfootnote{Author names are sorted alphabetically.}

\author{Qiang Du\thanks{Department of Applied Physics and Applied Mathematics, and Data Science Institute, Columbia University.}
	\and Kaizheng Wang\thanks{Department of Industrial Engineering and Operations Research, and Data Science Institute, Columbia University.}
	\and Edith Zhang\thanks{Department of Applied Physics and Applied Mathematics, Columbia University.}
	\and Chenyang Zhong\thanks{Department of Statistics, Columbia University.}
}

\date{December 2025}

\maketitle

\begin{abstract}
Variational inference is a fast and scalable alternative to Markov chain Monte Carlo and has been widely applied to posterior inference tasks in statistics and machine learning. A traditional approach for implementing mean-field variational inference (MFVI) is coordinate ascent variational inference (CAVI), which relies crucially on parametric assumptions on complete conditionals. We introduce a novel particle-based algorithm for MFVI, named PArticle VI (PAVI), for nonparametric mean-field approximation. We obtain non-asymptotic error bounds for our algorithm. To our knowledge, this is the first end-to-end guarantee for particle-based MFVI.
\end{abstract}


\section{Introduction}\label{sec: Introduction}

Variational inference (VI) is a popular Bayesian inference method for approximating intractable posterior distributions \citep{jordan1999introduction, wainwright2008graphical, blei2017variational}. The approximation is obtained by solving a constrained optimization problem:
\begin{equation}\label{eq: vi optimization problem}
	\min_{ q\in \cQ} \KL(q \| \pi) \,,
\end{equation} 
where $\pi \in \mathcal{P}(\R^n)$ is the posterior distribution, $\cQ \subseteq \mathcal{P}(\R^n)$ is some set of probability distributions, and $\KL(\cdot\|\cdot)$ is the Kullback--Leibler divergence. A common choice for $\cQ$ is the set of product measures, i.e., $\cQ = \big\{\prod_{i=1}^n q_i, \; q_i \in \mathcal{P}(\R)\big\}$, which is called the \textit{mean-field family}. The corresponding variational problem for this choice of $\cQ$ is called mean-field variational inference (MFVI). Compared to Markov chain Monte Carlo (MCMC) methods, which can have long burn-in periods, the optimization approach of VI allows for faster and more scalable inference in practice \citep{blei2017variational}.

Optimization problems over spaces of probability distributions can in theory be solved with a gradient flow by endowing the space of probability distributions with the 2-Wasserstein metric $\mathcal{W}_2$ and using a discrete scheme \citep{ambrosio2005gradient, santambrogio2017euclidean}. The use of $\mathcal{W}_2$--gradient flows to solve the variational problem \eqref{eq: vi optimization problem} has led to developments in theoretical understanding of VI \citep{lambert2022variational, yao2022mean, diao2023forward, Lac23} as well as the unconstrained version of \eqref{eq: vi optimization problem} whose solution is the exact posterior \citep{garcia2020bayesian}. 
Furthermore, a seminal work \cite{jordan1998variational} introduces a close connection between $\mathcal{W}_2$--gradient flows and PDEs: the discrete iterative scheme that describes the $\mathcal{W}_2$--gradient flow converges, in the small step-size limit, to the solution to a Fokker--Planck equation \cite[Theorem 5.1]{jordan1998variational}. The solution to the Fokker--Planck equation, in turn, describes the law of a particle evolving according to a stochastic differential equation (SDE). The differential equation characterization of the $\mathcal{W}_2$--gradient flow opens up avenues for computation and theoretical analysis.

Since VI operates in a subset $\cQ$ of the $\mathcal{W}_2$ space, efforts have been made to describe gradient flows for different choices of $\cQ$. One challenge is to prove that the $\mathcal{W}_2$--gradient flow constrained to $\cQ$ is still a gradient flow. 
When this is shown, theoretical guarantees follow \citep{lambert2022variational,yao2022mean, Lac23}.
In the case of MFVI, a recent paper \cite{Lac23} constructs a $\mathcal{W}_2$--gradient flow constrained to the submanifold of product measures. This gradient flow is the solution to a nonlinear Fokker--Planck equation, which is a system of coupled PDEs. The SDE corresponding to the PDEs is a McKean--Vlasov equation obtained through independent projections of diffusions. The paper also specifies the convergence rate of the gradient flow to the MFVI solution.

\paragraph{Contributions}

In this paper, we introduce a particle-based algorithm for MFVI, which we term PArticle VI (PAVI). The algorithm is built upon the SDE description of MFVI and stochastic approximation.
Under standard regularity assumptions, we provide an upper bound on the 2-Wasserstein distance between the particle approximation and the solution to \eqref{eq: vi optimization problem}. Notably, our convergence guarantee is non-asymptotic and holds for any number of particles and total time duration. To our knowledge, this is the first end-to-end guarantee for particle-based MFVI.

\paragraph{Related work} \label{sec: related work} 

We outline recent theoretical advances in VI based on gradient flows in the $\mathcal{W}_2$ space as follows. 
When $\cQ$ is the set of Gaussians, the resulting form of VI is called Gaussian VI. The set of Gaussians endowed with the $\mathcal{W}_2$ distance is also called the Bures--Wasserstein (BW) space \citep{bures1969extension}, which is a subset of the $\mathcal{W}_2$ space. By studying $\mathcal{W}_2$--gradient flows constrained to the BW space, \cite{lambert2022variational} propose a stochastic gradient descent algorithm for Gaussian VI---which discretizes the gradient flow---and provide theoretical guarantees, and \cite{diao2023forward} develop a forward-backward algorithm along with convergence rates for Gaussian VI. These works suggest extensions to the case when $\cQ$ is the set of mixtures of Gaussians, which corresponds to an essentially nonparametric $\cQ$. See \cite{yi2023bridging, huix2024theoretical} for further developments.

In the mean-field case where $\cQ$ consists of product distributions, \cite{yao2022mean} establish convergence guarantees for a coordinate-wise $\mathcal{W}_2$--gradient flow for MFVI in Bayesian models, and they propose two algorithms for practical implementation: a particle algorithm based on the Langevin SDE, and a function approximation method using neural networks. Their particle algorithm suffers from cumulative numerical errors that persist in the long-time limit, see Lemma H.1 and Theorem D.2 therein. Also, the error incurred by particle approximation is not analyzed. \cite{Lac23} shows that MFVI can be described by a $\mathcal{W}_2$--gradient flow constrained to the submanifold of product distributions and establishes continuous-time convergence guarantees. The difficulty in developing an efficient algorithm is also discussed. \cite{JCP23} propose the approximation of the MFVI solution through optimization over polyhedral subsets of the Wasserstein space, and introduces projected gradient descent algorithms with theoretical guarantees. Aside from those algorithmic and computational developments, there has also been recent work on the statistical accuracy of the solution to the MFVI problem for approximating the true posterior $\pi$. We refer the reader to \citep{zhang2020convergence, mukherjee2022variational, lacker2024mean} and the references therein for details. 

Our convergence result resembles those from the literature on Langevin Monte Carlo (LMC), see \cite{dalalyan2017theoretical, durmus2017nonasymptotic, cheng2018convergence, durmus2019analysis, balasubramanian2022towards}. LMC simulates an SDE whose stationary distribution is proportional to $e^{-V(\cdot)}$, where $V$ is a given potential. Non-asymptotic theoretical guarantees for LMC are given under the condition that the potential $V$ is continuously differentiable, strongly convex, and has a Lipschitz gradient. We refer the reader to \cite{chewi2024} for a survey.

Our algorithm is built on the framework of gradient flows and particle algorithms. Gradient flows and particle algorithms have recently been applied to a wide range of problems in statistics and optimization. For example, \cite{saeedi2017variational} propose a particle algorithm for approximating discrete distributions with the approximating family given by distributions with at most $K$ atoms (where $K$ is fixed). \cite{cheng2024particle} propose particle algorithms for distribution approximation using generalized Wasserstein gradient flow of the Kullback--Leibler divergence. Their algorithms do not impose explicit specification of an approximating family $\mathcal{Q}$ and rely on neural networks to estimate certain vector fields used for updating. Their theoretical analysis relies on assumptions about neural network approximation accuracy and does not consider the error arising from particle approximation. We also refer the reader to \cite{liu2016stein, gallouet2017jko, chizat2018global, liu2019understanding, lu2019accelerating, korba2020non, wang2020information, wang2022accelerated, lu2023birth, kook2024sampling,  liu2024towards, yan2024learning, YCY} and the references therein for further developments.  

\paragraph{Outline} In Section \ref{sec: MFVI}, we state the problem setup and discuss some relevant properties of MFVI. In Section \ref{sec: a particle algo}, we lay out the McKean--Vlasov equation corresponding to the MFVI problem, and introduce a particle algorithm. In Section \ref{sec: theoretical analysis}, we state our main theorem, which provides a convergence guarantee for our algorithm. In Section \ref{sec: proofs}, we prove the theoretical results stated in Section \ref{sec: theoretical analysis}. Some technical lemmas used in the proof are given in Section \ref{sec: technical lems}. Section \ref{sec: Discussions} is devoted to discussions.

\paragraph{Notation}

For any $d\in\mathbb{Z}_{+}=\{1,2,\cdots\}$, we denote $[d]=\{1,2,\cdots,d\}$ and let $\bI_d$ represent the identity matrix of size $d$. For any $d\in\mathbb{Z}_{+}$, we denote by $\mathcal{P}(\mathbb{R}^d)$ the set of probability distributions on $\mathbb{R}^d$, and by $\Pac(\RR^{d})$ the family of probability distributions on $\RR^{d}$ that are absolutely continuous with respect to the Lebesgue measure. We denote the Euclidean norm by $\|\cdot\|_2$. For any two symmetric matrices $\bA$ and $\bB$ of the same size, we write $\bA \preceq \bB$ or $\bB \succeq \bA$ if $\bB-\bA$ is positive semidefinite. For any random variable $X$, we denote its law by $\Law(X)$. For any $\bmu \in\mathbb{R}^d$ (where $d\in\mathbb{Z}_{+}$) and $d\times d$ positive definite matrix $\bSigma$, we denote by $N(\bmu,\bSigma)$ the multivariate normal distribution with mean vector $\bmu$ and covariance matrix $\bSigma$. For any two quantities $A, B>0$, we write $A=O_K(B)$ if $A\leq C B$ for a positive constant $C$ that only depends on $K$. For any function $f:\mathbb{R}^d\rightarrow\mathbb{R}$ (where $d\in\mathbb{Z}_{+}$), we denote its gradient and Hessian matrix by $\nabla f$ and $\nabla^2 f$, respectively, and denote $\partial_i f(\bx):=\frac{\partial f(\bx)}{\partial x_i}$ for any $\bx =(x_1,\cdots,x_d)\in\mathbb{R}^d$ and $i\in [d]$.

\section{Mean-field variational inference}\label{sec: MFVI}

In this section, we introduce the optimization problem underlying mean-field variational inference and discuss its relevant properties.

\subsection{Problem setup}

We define the \textit{mean-field family} of probability distributions as
\begin{align}
	\cQ
	= \bigg\{ q^1 \otimes \cdots \otimes q^m :~ q^i \in \Pac ( \RR ) ,~~\forall  i \in [m] \bigg\}  \,.
	\label{eqn-Q}
\end{align}
We see that $\cQ \subseteq \Pac ( \RR^{m} )$ consists of product distributions. Given a distribution $p \in \Pac ( \RR^{m} )$ whose density function is known up to a normalizing constant, \emph{mean-field variational inference} (MFVI) aims to approximate $p$ by a product distribution in $\cQ$, by solving the problem
\begin{align}
	\min_{ q \in \cQ } \KL ( q \| p ) \,.
	\label{eqn-MFVI}
\end{align}
Here $\KL ( q \| p ) = \EE_{ q } \log ( \frac{\rd q}{\rd p} ) $ is the Kullback--Leibler divergence (also known as the relative entropy) from $p$ to $q$. As a convention, $\KL ( q \| p )  = + \infty$ if $q$ is not absolutely continuous with respect to $p$. An optimal solution to \eqref{eqn-MFVI} is called a \emph{mean-field approximation} of $p$. MFVI is usually implemented using coordinate ascent variational inference (CAVI) \citep{bishop2006pattern, blei2017variational}. In addition to the mean-field assumption, CAVI requires each complete conditional of the underlying model to belong to an exponential family. To improve the expressive power of the variational family, nonparametric variants of CAVI have been proposed. For instance, $\mathcal{Q}$ can be the set of Gaussian mixtures \citep{gershman2012nonparametric, lambert2022variational}. In this paper, we work with the entire mean-field family \eqref{eqn-Q}.

\begin{remark}[Block MFVI]\label{remark-block}
	Prior works on MFVI have also used products of multivariate distributions to form \emph{block mean-field families} \citep{bishop2006pattern}. Given positive integers $\{ d_i \}_{i=1}^m$ and a target distribution $p \in \Pac ( \RR^{d_1 + \cdots + d_m} )$, block MFVI approximates $p$ by an optimal solution to the problem \eqref{eqn-MFVI} with
	\begin{align}
		\cQ = \bigg\{ q^1 \otimes \cdots \otimes q^m :~ q^i \in \Pac ( \RR^{d_i} ) ,~~\forall  i \in [m] \bigg\}  \,.
		\label{eqn-Q-block}
	\end{align}
	While our main discussion focuses on the mean-field family \eqref{eqn-Q} for notational simplicity, we will also present algorithmic and theoretical results for the general scenario.
\end{remark}


\subsection{Key properties of MFVI}\label{subsec: properties of MFVI}

In the following, we discuss several key properties of the optimization problem \eqref{eqn-MFVI}. Throughout the rest of this paper, we assume that the target distribution $p \in \Pac (\RR^{m})$ has density $Z^{-1}e^{-V(\cdot)}$, where $V:\mathbb{R}^m\rightarrow \mathbb{R}$ is a potential function and $Z$ is a normalizing constant.

We first introduce the following definitions. 

\begin{definition}[Entropy]
	For any $d \in \ZZ_+$ and $\mu \in \Pac (\RR^{d})$, we define the entropy of $\mu$ as $     H ( \mu ) = - \int_{ \RR^{d} } \mu (\bx) \log \mu (\bx) \rd \bx$.
\end{definition}

\begin{definition}[Marginals]
	For any $\bx = (x_1,\cdots, x_m) \in \RR^{m}$ and $i\in [m]$, define
	\begin{equation*}
		\bx_{-i} = ( x_1, \cdots, x_{i - 1}, x_{i + 1}, \cdots, x_m ).
	\end{equation*}
	For any $q =q^1 \otimes \cdots \otimes q^m \in \cQ$ and $i\in[m]$, define 
	\begin{equation*}
		q^{-i}  = q^1 \otimes \cdots q^{i-1} \otimes q^{i+1} \otimes \cdots q^m.
	\end{equation*}
\end{definition}
In words, the notation $-i$ means ``all components except for the $i$-th one.'' 
\begin{definition}[Approximate potential]
	For any $\mu \in \cP (\RR^{m-1})$ and $i\in[m]$, let
	\begin{equation}\label{eq: def V bar}
		\bar{V}_i ( \, \cdot \,, \mu ) =  \EE_{ \bx_{-i} \sim \mu }  V (  x_1, \cdots, x_{i - 1} , \, \cdot \,, x_{i + 1}, \cdots, x_m ) .
	\end{equation}
\end{definition}
The function $\bar{V}_i$ is an approximation (using a single component and a distribution $\mu$) to the potential $V: \mathbb{R}^m \rightarrow \mathbb{R}$. It will be crucial for defining the mean-field particle dynamics. In this paper, any derivative of $ \bar{V}_i$ is taken with respect to its first argument.

For any distribution $q\in\Pac(\RR^m)$, the Kullback--Leibler divergence from $p$ to $q$ can be rewritten as
\begin{align*}
	& \KL ( q \| p ) =  \EE_{ \bx \sim q } \log \bigg( \frac{\rd q}{\rd p} (\bx) \bigg)  = \EE_{\bx \sim q} V ( \bx ) - H (q) +\log(Z).
\end{align*}
Note that for a product distribution $q\in \mathcal{Q}$, we have $\EE_{\bx \sim q} V ( \bx ) =  \EE_{x_i \sim q^i}  \bar{V}_i ( x_i , q^{-i} )$ and $ H(q) = H( q^i ) + H( q^{-i} )$. Therefore, for $q\in \mathcal{Q}$, the Kullback--Leibler divergence can be written as a function of the $i$-th component $q^i$ as 
\begin{align*}
	\KL ( q \| p ) = \EE_{x_i \sim q^i} \bar{V}_i  ( x_i , q^{-i} ) - H ( q^i ) - H ( q^{-i} ) +\log(Z)
	, \qquad \forall i \in [m]\, .
\end{align*}

For each $i \in [m]$, define the transform $\cT_i: \Pac(\RR^{m-1}) \rightarrow \Pac(\RR)$ as 
\begin{equation*}
	\cT_i(\mu) = \frac{ 
		e^{- \bar{V}_i  ( \, \cdot \,, \mu )  }}{\int_{ \RR } e^{- \bar{V}_i  ( x , \mu )  } \rd x} \,,
\end{equation*}
which is an approximation of $p$ when all but one of the components of the argument of $V$ are distributed according to $\mu$. Using this notation, we have
\begin{align}
	\KL ( q \| p ) = \KL \left( q^i \|  \cT_i (q^{-i}) \right) - H ( q^{-i} )+\log(Z)-\log\bigg(\int_{ \RR } e^{- \bar{V}_i  ( x , q^{-i} )  } \rd x\bigg) .
	\label{eqn-KL-coordinate}
\end{align}
If we minimize the Kullback--Leibler divergence $\KL ( q \| p )$ while keeping $q^{-i}$ fixed, then the optimal $q^i $ is clearly $\cT_i (q^{-i}) $. Hence, the optimality of 
\begin{equation*}
	q_* \in \argmin_{q \in \cQ} \KL (q \| p)
\end{equation*}
implies the optimality of the $i$-th component $q_*^i  = \cT_i ( q_*^{-i} )$ for every $ i \in [m]$. We immediately obtain the following fixed point characterization.

\begin{lemma}[Fixed point characterization]\label{lem-fixed-point}
	$q_*$ is a fixed point of transform $\cT:~ \cQ \to \cQ$, $q \mapsto \cT_1 (q^{-1}) \otimes \cdots \otimes \cT_m (q^{-m})$.
\end{lemma}

\section{A particle algorithm for mean-field variational inference}\label{sec: a particle algo}

In this section, we introduce a particle algorithm for MFVI, which we term PArticle VI (PAVI). In Section \ref{fp-diff}, we formulate the variational inference problem in terms of Fokker--Planck equations and the corresponding diffusion processes. Then, based on this formulation, we introduce the particle algorithm for practical implementation.

\subsection{Fokker--Planck equations and diffusion processes}\label{fp-diff}

We start by considering a sub-problem of \eqref{eqn-MFVI}: 
\begin{center}
	\emph{How to optimize $\KL (q \| p)$ with respect to $q^i$ while keeping $q^{-i}$ fixed?}
\end{center}
\noindent In light of \eqref{eqn-KL-coordinate}, this sub-problem amounts to solving for a univariate distribution $\cT_i (q^{-i}) \propto e^{- \bar{V}_i  ( \, \cdot \, , q^{-i} )  }$. We fix $q^{-i}$, and define
\begin{equation*}
	f (\, \cdot \, ) = \bar{V}_i  ( \, \cdot \, , q^{-i} )  \,.
\end{equation*}
According to the general results by \cite{jordan1998variational}, the gradient flow for the Kullback--Leibler divergence $\KL ( \, \cdot \, \| \cT_i (q^{-i}) )$ with respect to the 2-Wasserstein metric is the solution to the Fokker--Planck equation
\begin{equation}\label{eq: fpe}
	\begin{cases}
		& \partial_t \rho (x, t) = \partial_x [ f'(x)  \rho(x, t) ] + \partial^2_{xx} \rho(x, t) ,
		\\
		& \rho(\, \cdot \,, 0) = q^i (\, \cdot \,) .
	\end{cases}
\end{equation}
Below is a formal definition of the aforementioned metric, which will be used throughout our analysis.
\begin{definition}[2-Wasserstein metric]
	For any $d \in \ZZ_+$ and $\mu, \nu \in \cP (\RR^d)$, the {\bf 2-Wasserstein distance} between $\mu$ and $\nu$ is 
	\begin{align*}
		\cW_2 ( \mu, \nu ) = \inf_{\gamma \in \cC (\mu, \nu) } \sqrt{  \EE_{ ( \bx, \by )  \sim \gamma } \| \bx - \by \|_2^2 } \,,
	\end{align*}
	where $\cC (\mu, \nu)$ is the set of joint distributions on $\RR^d \times \RR^d$ with $\mu$ and $\nu$ as their marginals. Any $\gamma \in \cC (\mu, \nu)$ is called a coupling between $\mu$ and $\nu$. 
\end{definition}

The Fokker--Planck equation is deeply connected to a stochastic differential equation: the solution $\rho(\, \cdot \,, t)$ to \eqref{eq: fpe} is the law of the diffusion process
\begin{align}
	\begin{cases}
		& 	\rd Y_t = - f' ( Y_t  ) \rd t + \sqrt{2} \rd B_t , \qquad t \geq 0
		\\
		& Y_0 \sim q^i ,
	\end{cases}
\end{align}
where $\{ B_t \}_{t \geq 0}$ is a standard Brownian motion that is independent of $Y_0$. 
By simultaneously updating all the marginal distributions $\{ q^i \}_{i=1}^m$ according to the above diffusion, we get a McKean--Vlasov process: 
\begin{align}\label{eqn-MFVI-WGF}
	\begin{cases}
		& \rd Y_{t, i} = -  \bar{V}_i' ( Y_{t, i} , \Law ( Y_{t, -i} )  ) \rd t + \sqrt{2} \rd B_t^i , \qquad i \in [m],\, t \geq 0  \\
		& Y_0 \sim q^1 \otimes \cdots \otimes q^m \,,
	\end{cases}.
\end{align}
where $\bar{V}_i$ is defined as in \eqref{eq: def V bar}, $Y_t = ( Y_{t, 1} , \cdots, Y_{t, m} )$, and $\{ ( B_{t, 1}, \cdots , B_{t, m} ) \}_{t \geq 0}$ is an $m$-dimensional standard Brownian motion that is independent of $Y_0$. The trajectory $t \mapsto \Law ( Y_t ) $ is the Wasserstein gradient flow for optimizing $\KL ( \, \cdot \, \| p )$ over the mean-field family $\cQ$ \citep{Lac23}. Note that the process $\{ Y_t \}_{t \geq 0}$ has independent coordinates because the $m$-dimensional Brownian motion has independent coordinates and the other coordinates $Y_{t,-i}$ are integrated out in the drift term. 
It is also worth pointing out that if $Y_0 \sim q_*$, then $Y_t \sim q_*$ holds for all $t \geq 0$.

\subsection{PArticle VI}

We are ready to present the PArticle VI (PAVI), a particle algorithm for mean-field variational inference. To set the stage, we first introduce some convenient notations.

\begin{definition}[Empirical measure of an array of particles]\label{defn-particles}
	Let $S_N$ be the set of all permutations of $[N]$. For any array $\bA \in \RR^{m \times N}$, denote by $A^{i, j}$ its $(i, j)$-th entry; let $\bA^{i, :} = ( A^{i, 1} , \cdots , A^{i, N} )$ be the $i$-th row, $\bA^{:, j} = ( A^{1, j} , \cdots , A^{m, j} )$ be the $j$-th column, and $\bA^{-i, j} =  ( A^{1, j} , \cdots , A^{i-1, j} , A^{i+1, j} , \cdots , A^{m, j} )$ be the $j$-th column of $\bA$ with the $i$-th entry removed. For each $i \in [m]$, let $q_{\bA}^i = 	\frac{1}{N} \sum_{j = 1}^N \delta_{ A^{i,j} }$. Define $q_{\bA} = q_{\bA}^1 \otimes \cdots \otimes q_{\bA}^m$. 
\end{definition}

\begin{remark}
	For each $i\in[m]$, $q^i_{\bA}$ is the empirical measure of $\bA^{i,:}$. $q_{\bA} \in \cP (\RR^m)$ is the product measure with $\{  q^i_{\bA} \}_{i=1}^m$ as its marginals. Hence, it has up to $N^m$ atoms.
\end{remark}

Our first version of PAVI, \Cref{gd algo}, can be viewed as a direct discretization of the McKean--Vlasov process (\ref{eqn-MFVI-WGF}).

\begin{algorithm}[h] 
	{\bf Input:} Function $V:~\RR^m \to \RR$, initial value $\bX_0 \in \RR^{m \times N}$, step-size $h > 0$, number of iterations $T \in \ZZ_+$.\\
	{\bf For $n = 0,1,\ldots, T - 1$:}\\
	\hspace*{.6cm} {\bf For $i = 1,\cdots, m$:} 
	\\
	\hspace*{1.2cm} Draw a random vector $\bxi_{n}^{i,:}$ from $N ( 0, \bI_N )$, the $N$-dimensional standard normal distribution. \\
	\hspace*{1.2cm} {\bf For $j = 1,\cdots, N$:} \\
	\hspace*{1.8cm} Compute $X_{n+1}^{i,j} = X_{n}^{i,j} - h  \bar{V}_i' ( X_{n}^{ i,j } , q^{-i}_{ \bX_n }  ) + \sqrt{ 2 h } \xi_{n}^{ i,j } $.\\
	{\bf Output:} $ \bX_T $.
	\caption{A gradient descent version of PAVI}
	\label{gd algo}
\end{algorithm}

If the columns of $\bX_0$ are i.i.d.~samples from a probability distribution over $\RR^m$, $N \to \infty$ and $h \to 0$, then \Cref{gd algo} becomes the Wasserstein gradient flow in \cite{Lac23}. 
Meanwhile, we note that \Cref{gd algo} is intuitive but impractical: the gradient evaluation is prohibitively expensive due to the cost of averaging over $N^{m-1}$ terms. For computational reasons, we propose a second version of PAVI, presented in \Cref{sgd algo} below, which involves a stochastic approximation of the gradient term.

\begin{algorithm}[h] 
	{\bf Input:} Function $V:~\RR^m \to \RR$, initial value $\bX_0 \in \RR^{m \times N}$, step-size $h > 0$, batch size $B \in \ZZ_+$, number of iterations $T \in \ZZ_+$.\\
	{\bf For $n = 0,1,\ldots, T - 1$:}\\
	\hspace*{.6cm} 	Draw $B$ i.i.d.~samples $\{ \bz_n^{:, b} \}_{b=1}^B$ from $q_{ \bX_n }$.\\	
	\hspace*{.6cm} {\bf For $i = 1,\cdots, m$:}  
	\\
	\hspace*{1.2cm} Define a function $g_{n}^i ( \, \cdot \,) = \frac{1}{B} \sum_{b=1}^{B} \partial_i  V ( 
	z_{n}^{1,b}, \cdots, z_{n}^{( i-1),b}, \cdot, z_{n}^{(i+1),b}, \cdots, z_{n}^{m,b}
	)$.\\
	\hspace*{1.2cm} Draw a random vector $\bxi_{n}^{i,:}$ from $N ( \mathbf{0}, \bI_N )$, the $N$-dimensional standard normal distribution.\\
	\hspace*{1.2cm} {\bf For $j = 1,\cdots, N$:} \\
	\hspace*{1.8cm} Compute $X_{n+1}^{i,j} = X_{n}^{i,j} - h g_n^i ( X_{n}^{ i,j }   ) + \sqrt{ 2 h } \xi_{n}^{ i,j } $.\\
	{\bf Output:} $ \bX_{T} $.
	\caption{PAVI}
	\label{sgd algo}
\end{algorithm}

\Cref{gd algo} is the limiting case of \Cref{sgd algo} (PAVI) as $B \to \infty$. The finite batch size in PAVI ensures its computational efficiency. The output $\bX_T$ from \Cref{sgd algo} is associated with $q_{\bX_{T}}$, which is a product distribution that approximates the solution $q_*$ to (\ref{eqn-MFVI}). In the next section, we will prove the convergence of $q_{\bX_T}$ towards $q_*$ in the $\mathcal{W}_2$ distance.

\begin{remark}[Block PAVI]\label{remark-block-PAVI}
	To extend \Cref{sgd algo} for the block MFVI introduced in \Cref{remark-block}, one would use $\bX_n \in \RR^{(d_1 + \cdots + d_m) \times N}$ and operate on the $m$ blocks of coordinates: $X_n^{i, j}$ consists of the entries of $\bX_n$ in the $(1+ \sum_{k=1}^{i-1} d_k)$-th through $\sum_{k=1}^{i} d_k$-th rows and the $j$-th column. The partial derivative $\partial_i$ in the function $g_n^i$ should be replaced with the gradient with respect to the $i$-th block of arguments.
	
	One can even further generalize the above algorithm by using an irregular array $ \{ X_n^{i, j} \}_{i \in [m], j \in [N_i] }$, where each block has its own number of particles. This can be helpful when the marginals of the block mean-field approximation behave very differently.
\end{remark}

\section{Theoretical analysis}\label{sec: theoretical analysis}

We now provide theoretical guarantees for Algorithm \ref{sgd algo}, the particle algorithm with stochastic gradient approximation, and its extension to the block case. Recall from Section \ref{subsec: properties of MFVI} that the target distribution $p\in\Pac(\RR^m)$ has density $Z^{-1}e^{-V(\cdot)}$. We mainly work under the commonly used assumption below.

\begin{assumption}[Strong convexity and smoothness]\label{assumption-regularity}
	There exist constants $0 < \alpha \leq L < \infty$ such that for any $\bx\in\RR^m$, we have $\alpha \bI_m \preceq \nabla^2 V (\bx) \preceq L \bI_m$.
\end{assumption}

The following result shows the existence and uniqueness of the MFVI solution $q_*$, and establishes the exponential convergence of the law of $Y_t$ in \eqref{eqn-MFVI-WGF} towards $q_*$.

\begin{lemma}[Theorems 1.1 of \cite{lacker2024mean} and Theorem 2.3 of \cite{Lac23}]\label{lem-solution}
	Under Assumption \ref{assumption-regularity}, the optimization problem \eqref{eqn-MFVI} has a unique solution $q_*$, and
	\[
	\cW_2 ( \Law  (Y_t ) , q_* ) \leq e^{- \alpha t} \cW_2 ( \Law  (Y_0 ) , q_* ), \qquad \forall t \geq 0.
	\]
\end{lemma}

Lemma \ref{lem-solution} focuses on the long-time convergence of the continuous-time process $Y_t$ but does not provide algorithmic guarantees. Our main theoretical result, presented in Theorem \ref{thm-main} below, establishes the first end-to-end theoretical guarantee for an implementable particle algorithm (Algorithm \ref{sgd algo}) for MFVI. We provide a non-asymptotic upper bound on the $\mathcal{W}_2$ distance between $q_{\bX_n}$ (the product distribution defined by $N$ particles at the $n$-th iteration) and the MFVI solution $q_*$.
The proof will be presented in Section \ref{sec: proofs}.

\begin{theorem}\label{thm-main}
	Let Assumption \ref{assumption-regularity} hold. Run Algorithm \ref{sgd algo} with $N \geq 2$, $0 < h <  \alpha / (4 L^2) $ and $B \geq 1$. There exists a universal constant $C$ such that for any $n \in \NN$,
	\begin{align*}
		\sqrt{ \EE \cW_2^2 ( q_{\bX_n}, q_* ) }
		& \leq (1 - \alpha h/2)^{n/2} \sqrt{  \EE \cW_2^2 ( q_{\bX_0}, q_* ) } +  
		C \sqrt{m} \bigg(
		\sqrt{ \frac{ \log N }{\alpha N} }  +
		\frac{L  \sqrt{h} }{\alpha}
		\bigg).
	\end{align*}
\end{theorem}

Our non-asymptotic bound quantifies the convergence of PAVI. To the best of our knowledge, this is the first such guarantee for particle-based MFVI. The error comes from four sources:
\begin{itemize}
	\item Initialization: The impact of the initial discrepancy is bounded by $( 1 - \alpha h / 2 )^{n/2} \sqrt{ \EE	\cW_2^2   ( q_{\bX_0} , q_*) } $, which decays exponentially over time.
	
	\item Particle approximation: The error of discretizing the optimal solution $q_*$ by $N$ particles is $O( \sqrt{ \frac{ m \log N }{ \alpha N } } )$. This is sharp up to a logarithmic factor: if $N \geq 3$ and $\bY \in \RR^{m \times N}$ has independent columns drawn from $q = N( \mathbf{0}, \alpha^{-1} \bI_m )$, then $\sqrt{ \EE \cW_2^2 ( q_{\bY}, q ) } \geq \sqrt{ \frac{c m \log\log N}{\alpha N} }$ holds with some universal constant $c$. See \Cref{sec-particle-proof} for a proof.
	
	\item Time discretization: The finite difference (Euler-Maruyama) scheme incurs an error that is $ O ( \alpha^{-1} L \sqrt{ m h } )$. 
	
	\item Stochastic gradient approximation: Algorithm \ref{sgd algo} uses $B$ random samples to estimate the gradient in Algorithm \ref{gd algo}. For any $B \geq 1$, our proof shows that the stochastic error is $O( \alpha^{-1} L \sqrt{ m h } )$. This is dominated by the time discretization error and hence does not appear explicitly in the final inequality. One can safely use $B = 1$ for computational efficiency.
\end{itemize}

Our theory can be generalized to the block PAVI algorithm discussed in \Cref{remark-block-PAVI}. Suppose that the target distribution $p\in\Pac(\RR^d)$ has density $Z^{-1}e^{-V(\cdot)}$, and $\alpha \bI_d \preceq \nabla^2 V  \preceq L \bI_d$ holds for some $0 < \alpha \leq L < \infty$. Let $\{ d_i \}_{i=1}^m$ be positive integers that sum up to $d$. The variational problem \eqref{eqn-MFVI} over the block mean-field family \eqref{eqn-Q-block} still has a unique solution $q_* = q_*^1 \otimes \cdots \otimes q_*^m$ \citep{Lac23}. 
We have the following error bound. The proof can be found in \Cref{sec-thm-block-proof}.

\begin{theorem}\label{thm-block}
	Under the above assumptions, the block version of Algorithm \ref{sgd algo} with $N \geq 2$, $0 < h <  \alpha / (4 L^2) $ and $B \geq 1$ satisfies
	\begin{align*}
		\sqrt{ \EE \cW_2^2 ( q_{\bX_n}, q_* ) }
		& \leq (1 - \alpha h/2)^{n/2} \sqrt{  \EE \cW_2^2 ( q_{\bX_0}, q_* ) } 
		+  
		\sqrt{\sum_{i=1}^{m} \EE \cW_2^2 ( \hat{q}_N^i, q_*^i )
		}
		+
		\frac{4 L \sqrt{h d } }{\alpha} ,
	\end{align*}
	where $ \hat{q}_N^i$ is the empirical distribution formed by $N$ independent samples from $q_*^i$.
\end{theorem}

When $d_1 = \cdots = d_m = 1$, standard concentration inequality of univariate empirical measure (e.g.,~\Cref{lem-W2-concentration}) shows that $\EE \cW_2^2 ( \hat{q}_N^i, q_*^i ) = O( \frac{\log N}{\alpha N} )$. Then, \Cref{thm-block} reduces to \Cref{thm-main}. For general $d_i$'s, one can bound $\EE \cW_2^2 ( \hat{q}_N^i, q_*^i ) $ using the results in \cite{FGu15}.

We now come back to the original version of \Cref{sgd algo}. \Cref{thm-main} suggests taking $B = 1$ for computational efficiency, and making $\frac{L }{\alpha} \sqrt{h}$ and $\sqrt{ \frac{ \log N }{\alpha N} } $ have the same order. Based on these observations, we get the following user-friendly error bound. See \Cref{sec-corollary-main-proof} for the proof.

\begin{corollary}\label{corollary-main}
	Let Assumption \ref{assumption-regularity} hold. Run Algorithm \ref{sgd algo} with $N \geq 9$, $h = \frac{ \alpha \log N }{ L^2 N }$ and $B = 1$. Suppose that the columns of $\bX_0 \in \RR^{m \times N}$ are drawn independently from a product distribution $q_0$, whose marginals are log-concave with variance bounded by $\alpha$. There exists a universal constant $C>0$ such that for any $n \in \NN$,
	\begin{align*}
		\sqrt{ \EE	\cW_2^2   ( q_{\bX_n} , q_*) }  & \leq 
		\exp
		\bigg(
		- \frac{ n \alpha^2 \log N }{ 4 L^2 N }
		\bigg) 	\cW_2 ( q_{0} , q_*)  + C \sqrt{
			\frac{m \log N}{\alpha N}
		}.
	\end{align*}
\end{corollary}

Suppose that $\alpha$ and $L$ are constants. Choose any $\varepsilon \in (0, 1)$. There exists a constant $K$ determined by $\alpha$ and $L$, such that when $N > K m \varepsilon^{-2} \log ( m \varepsilon^{-2} )$, we have $C \sqrt{
	\frac{m \log N}{\alpha N}
} \leq \varepsilon/ 2$. Meanwhile, when
\[
n \geq \frac{4L^2 N}{ \alpha^2 \log N } \log [
2 \varepsilon^{-1}  \cW_2 ( q_{0} , q_*) 
],
\]
we have $\exp
(
- \frac{ n \alpha^2 \log N }{ 4 L^2 N }
) \cW_2 ( q_{0} , q_*)   \leq \varepsilon / 2 $.
Therefore, taking
\[
N = O_{\alpha, L}
\Big(
m \varepsilon^{-2} \log ( m \varepsilon^{-2} )
\Big)
\quad\text{and}\quad
n = O_{\alpha, L} \Big(
m \varepsilon^{-2}
\log [
2 \varepsilon^{-1} \cW_2 ( q_{0} , q_*) 
]
\Big)
\]
with sufficiently large constant factors can guarantee $\sqrt{ \EE	\cW_2^2   ( q_{\bX_n} , q_*) } \leq  \varepsilon $. Both $N$ and $n$ have linear dependence on $m$ up to logarithmic factors.

We conclude this section with a comparison between theoretical guarantees for Algorithm \ref{sgd algo} and Langevin Monte Carlo (LMC). 
\begin{itemize}
	\item Our algorithm maintains $N$ interacting particles. According to Corollary \ref{corollary-main}, it directly approximates the MFVI solution $q_{*}$ up to $O( \sqrt{ \frac{m \log N}{N} } )$ in $\cW_2$. The fast convergence of particle approximation could compensate for the discrepancy between $q_*$ and the target distribution $p$.
	
	\item LMC tracks the evolution of a single particle. Theoretical results \citep{dalalyan2017theoretical, durmus2017nonasymptotic, cheng2018convergence, durmus2019analysis} quantify how its distribution converges to $p$. 
	In order to approximate $p$, one generates a large number of such particles (e.g.,~by running $N$ independent replicates of LMC) and uses their empirical distribution.
	This would suffer from the curse of dimension: if we only assume Assumption \ref{assumption-regularity}, then the empirical distribution formed by $N$ independent particles converges in $\mathcal{W}_2$ at a rather slow rate of $N^{-1/m}$ when the dimension $m$ exceeds 3 \citep{FGu15}. 
\end{itemize}

\section{Proofs}\label{sec: proofs}

We present the proofs of our main results.

\subsection{Proof of \Cref{thm-main}}\label{sec-main-proof}

We begin with a recursive bound based on a coupling between the iterates $\{ \bX_n \}_{n=0}^{\infty}$ of \Cref{sgd algo} and a stationary process. The proof is deferred to \Cref{sec-lem-main-recursion-proof}.


\begin{lemma}\label{lem-main-recursion}
	Suppose that Assumption \ref{assumption-regularity} holds and $0 < h <  \min \{ \frac{ 2 }{ \alpha + L } , \frac{B \alpha}{4L^2} \} $. 
	There exists a stationary process $\{ \bY_t \}_{t \geq 0}$ with $\bY_t \in \RR^{m\times N}$, such that for any $t \geq 0 $, $\{ \bY_t^{:,j} \}_{j=1}^N$ are i.i.d.~samples from $q_*$, and
	\begin{align*}
		\EE \cW_2^2 ( q_{\bX_{n+1}} , q_{\bY_{(n+1)h}}  )
		\leq
		( 1 - \alpha h /2 ) \EE \cW_2^2 ( q_{\bX_n}, q_{\bY_{nh}} ) 
		+  \frac{ 4 m h^2 L^2 }{\alpha} \bigg(
		1 + \frac{1}{2B} + \frac{h L^2 }{\alpha}
		\bigg)
		.
	\end{align*}
\end{lemma}

Define $W_n = \EE \cW_2^2 ( q_{\bX_n}, q_{\bY_{nh}} ) $. When $0 < h < \alpha / (4 L^2)$, \Cref{lem-main-recursion} implies that $W_{n+1} \leq (1-A) W_n + R$, where $A = \alpha h /2$ and $R = 8 m h^2 L^2 / \alpha$. By induction, we get
\begin{align*}
	W_{n} \leq (1-A)^n W_0 + \sum_{k=0}^{n-1} (1-A)^k R
	\leq (1-A)^n W_0 + \frac{R}{A}
\end{align*}
and thus $\sqrt{W_n} \leq (1-A)^{n/2} \sqrt{W_0} + \sqrt{R/A}$. Then,
\begin{align}
	\sqrt{W_n} \leq (1-A)^{n/2} \sqrt{W_0} + \sqrt{R/A}
	=  (1 - \alpha h/2)^{n/2} \sqrt{W_0} +  \frac{ 4 L }{\alpha} \sqrt{mh} .
	\label{eqn-proof-main-1}
\end{align}

We now relate $W_n$ to $\EE \cW_2^2 ( q_{\bX_n}, q_* )$. By \Cref{lem-W2-triangle}, we have
\begin{align*}
	& \sqrt{ \EE	\cW_2^2   ( q_{\bX_n} , q_{\bY_{nh}}) } 
	\geq \sqrt{ \EE	\cW_2^2   ( q_{\bX_n} , q_*) } - \sqrt{ \EE	\cW_2^2   ( q_{\bY_{nh}} , q_*) }  , \\
	& \sqrt{ \EE	\cW_2^2   ( q_{\bX_0} , q_{\bY_{0}}) } 
	\leq \sqrt{ \EE	\cW_2^2   ( q_{\bX_0} , q_*) }  + \sqrt{ \EE	\cW_2^2   ( q_{\bY_{0}} , q_*) }  .
\end{align*}
The stationarity of $\{ \bY_t \}_{t \geq 0}$ implies $ \EE	\cW_2^2   ( q_{\bY_{nh}} , q_*) =  \EE	\cW_2^2   ( q_{\bY_{0}} , q_*)$. By the above estimates and \eqref{eqn-proof-main-1},
\begin{align}
	\sqrt{ \EE \cW_2^2 ( q_{\bX_n}, q_* ) }
	& \leq (1 - \alpha h/2)^{n/2} \sqrt{  \EE \cW_2^2 ( q_{\bX_0}, q_* ) } +  \frac{ 4 L }{\alpha} \sqrt{mh} 
	+ 2 \sqrt{ \EE	\cW_2^2   ( q_{\bY_{0}} , q_*) }  .
	\label{eqn-thm-main-proof}
\end{align}
Note that the columns of $\bY_0$ are i.i.d.~samples from $q_*$. By \Cref{lem-grad-square}, each $q_*^i$ is log-concave with variance bounded by $\alpha^{-1}$. Then, \Cref{cor-W2-concentration} implies that $ \EE	\cW_2^2   ( q_{\bY_{0}} , q_*) = O(  \frac{ m \log N}{\alpha N} ) $. This completes the proof of \Cref{thm-main}. It remains to prove \Cref{lem-main-recursion}, which can be found in \Cref{sec-lem-main-recursion-proof}.

\subsection{Proof of the claim regarding the particle approximation error}\label{sec-particle-proof}

Note that for $i \in [m]$, $ q_{\bY}^i $ is an empirical measure formed by $N$ independent samples from the marginal distribution $q^i = N ( 0, \alpha^{-1} )$. Corollary 6.14 in \cite{BLe19} shows the existence of a universal constant $c$ such that $\EE \cW_2^2 ( q_{\bY}^i , q^i ) \geq \frac{c \log \log N}{\alpha N}$ holds for all $N \geq 3$. By \Cref{lem-additivity-W2}, we have
\[
\EE \cW_2^2 ( q_{\bY} , q ) = \sum_{i=1}^{m} \EE \cW_2^2 ( q_{\bY}^i , q^i ) \geq  \frac{c m \log \log N}{\alpha N}.
\]

\subsection{Proof of \Cref{thm-block}}\label{sec-thm-block-proof}

Examining the proof of \Cref{thm-main}, we see that the derivation of \eqref{eqn-thm-main-proof} works seamlessly for the block case. The result now reads
\begin{align*}
	\sqrt{ \EE \cW_2^2 ( q_{\bX_n}, q_* ) }
	& \leq (1 - \alpha h/2)^{n/2} \sqrt{  \EE \cW_2^2 ( q_{\bX_0}, q_* ) } +  \frac{ 4 L }{\alpha} \sqrt{d h} 
	+ 2 \sqrt{ \EE	\cW_2^2   ( q_{\bY_{0}} , q_*) }  ,
\end{align*}
where $\bY_0 \in \RR^{d \times N}$ is a random array whose columns are i.i.d.~samples from $q_*$, and $q_{\bY_0}$ is constructed according to the block version of \Cref{defn-particles}. More concretely, let $Y_0^{i, j}$ be the entries of $\bY_0$ in the $(1+ \sum_{k=1}^{i-1} d_k)$-th through $\sum_{k=1}^{i} d_k$-th rows and the $j$-th column. Then, $q_{\bY_0}$ is defined by $q_{\bY_0}^1 \otimes \cdots \otimes q_{\bY_0}^m$ with $q_{\bY_0}^i = 	\frac{1}{N} \sum_{j = 1}^N \delta_{ Y_0^{i,j} }$. By \Cref{lem-additivity-W2}, $\cW_2^2 ( q_{\bY_0}, q_* )
=
\sum_{i=1}^{m} \cW_2^2 ( q_{\bY_0}^i, q_*^i )$.
We complete the proof using the fact that $q_{\bY_0}^i $ is an empirical distribution formed by $N$ independent samples from $q_*^i$.

\subsection{Proof of \Cref{corollary-main}}\label{sec-corollary-main-proof}
When $N \geq 9$, we have $N^{-1} \log N < 1/4$ and $h = \frac{ \alpha \log N }{ L^2 N } < \alpha / (4 L^2)$. Hence, \Cref{thm-main} holds. Since $1 - x \leq e^{-x}$ for all $x \in \RR$, we have $(1-\alpha h / 2)^{n/2} \leq e^{-\alpha h n / 4}$ and
\begin{align*}
	\sqrt{ \EE \cW_2^2 ( q_{\bX_n}, q_* ) }
	& \leq 
	\exp
	\bigg(
	- \frac{ n \alpha^2 \log N }{ 4 L^2 N }
	\bigg) 	
	\sqrt{  \EE \cW_2^2 ( q_{\bX_0}, q_* ) } +  
	2	C  \sqrt{ \frac{m  \log N }{\alpha N} } .
\end{align*}
By \Cref{lem-W2-triangle}, we have
\[
\sqrt{  \EE \cW_2^2 ( q_{\bX_0}, q_* ) } \leq 
\sqrt{  \EE \cW_2^2 ( q_{\bX_0}, q_0 ) } + \cW_2 ( q_0 , q_* ).
\]
\Cref{cor-W2-concentration} guarantees $\EE \cW_2^2 ( q_{\bX_0}, q_0 ) = O( \frac{m \log N}{\alpha N} )$. Then, we obtain the desired result by redefining $C$.

\subsection{Proof of \Cref{lem-main-recursion}}\label{sec-lem-main-recursion-proof}

The proof is based on a coupling argument.

\subsubsection{Step 1: Construction of coupling}

To begin with, let $\bX_0 \in \RR^{m \times N}$ be the initial value of Algorithm \ref{sgd algo} that is allowed to be random; $\{ B_t^{i, j} \}_{t \geq 0}$, $(i, j) \in [m] \times [N]$ be i.i.d.~standard Brownian motions that are independent of $\bX_0$. By running Algorithm \ref{sgd algo} with $\bxi_{n}^{i, :}$ replaced by $( \bB_{(n+1)h}^{i, :} - \bB_{nh}^{i, :} ) / \sqrt{h}$, we obtain a sequence $\{ \bX_n \}_{n=0}^{\infty}$ that has the same distribution as the one in Algorithm \ref{sgd algo}. 

Next, we construct an auxiliary process $\{ \bY_t \}_{t \geq 0}$ that is coupled with $\{ \bX_n \}_{n=0}^{\infty}$. Let $\bY_0 \sim q_*$ be independent of the above construction of $\{ \bX_n \}_{n=0}^{\infty}$. Choose any $n \in \NN$ and suppose that $\{ \bY_t \}_{0 \leq t \leq n h}$ has been defined. We now define $\{ \bY_t \}_{nh < t \leq (n+1) h}$.
For each $i \in [m]$, the Birkhoff--von Neumann algorithm implies the existence of a permutation $\tau_{n, i}$ of $[N]$, determined by $\bX_n^{i, :}$ and $\bY_{nh}^{i, :}$, such that
\begin{align}
	\cW_2^2 ( q_{\bX_n}^i , q_{\bY_{nh}}^i ) = \frac{1}{N} \sum_{j=1}^N | X_n^{i, j} - Y_{nh}^{i, \tau_{n,i}(j) }  |^2 .
	\label{eqn-permutation}
\end{align}
See Chapter 2 of \cite{Vil09}. Define $\{ \bY_t \}_{nh \leq t \leq (n+1) h}$ as the strong solution to the SDE
\begin{align}
	& \rd Y_t^{i,j} = -  \bar{V}_i' ( Y_t^{ i,j } , q_*^{-i} ) \rd t + \sqrt{2} \rd B_t^{i, \tau_{n, i}^{-1} (j) }, \qquad nh \leq t \leq (n+1)h.
	\label{eqn-Y-new}
\end{align}
Let $\bar{B}_0^{i, j} = 0$ for all $(i, j) \in [m] \times [N]$. For every $n \in \NN$ and $t \in ( nh, (n+1)h ]$, define $\bar{B}_t^{i, j} = 
\bar{B}_{nh}^{i, j} + ( B_t^{i, \tau_{n, i}^{-1} (j) } - B_{nh}^{i, \tau_{n, i}^{-1} (j) } )$. Then, $\{ \bar{B}_t^{i, j} \}_{t \geq 0}$, $(i, j) \in [m] \times [N]$ are i.i.d.~standard Brownian motions that are independent of $\bY_0$. Our process $\{ \bY_t \}_{t \geq 0}$ solves the SDE
\begin{align*}
	\rd Y_t^{i,j} = -  \bar{V}_i' ( Y_t^{ i,j } , q_*^{-i} ) \rd t + \sqrt{2} \rd \bar{B}_t^{i,j }.
\end{align*}
Hence, for any $t \geq 0 $, $\{ \bY_t^{:,j} \}_{j=1}^N$ are i.i.d.~samples from $q_*$. Below we will analyze $\{ \bX_n \}_{n =0}^{\infty}$ and $\{ \bY_t \}_{t \geq 0}$.

\subsubsection{Step 2: Error decomposition}

Choose any $n \in \NN$. By \eqref{eqn-permutation} and \Cref{lem-additivity-W2},
\begin{align}
	\cW_2^2 ( q_{\bX_n}, q_{\bY_{nh}} ) = \sum_{i=1}^m \cW_2^2 ( q_{\bX_n}^i , q_{\bY_{nh}}^i ) 
	= \frac{1}{N} \sum_{ (i,j) \in [m] \times [N] } | X_n^{i, j} - Y_{nh}^{i, \tau_{n,i}(j) }  |^2 .
	\label{eqn-W2-frobenius-0}
\end{align}
Meanwhile, \Cref{lem-additivity-W2} and the definition of $\cW_2$ imply that
\begin{align}
	\cW_2^2 ( q_{\bX_{n+1} }, q_{\bY_{(n+1)h}} ) 
	& = \sum_{i=1}^m \cW_2^2 ( q_{\bX_{n+1}}^i , q_{\bY_{(n+1)h}}^i ) \notag\\
	& \leq \frac{1}{N} \sum_{ (i,j) \in [m] \times [N] } | X_{n+1}^{i, j} - Y_{(n+1)h}^{i, \tau_{n,i}(j) }  |^2 .
	\label{eqn-W2-frobenius}
\end{align}
We will show how the error evolves after one iteration. Note that
\begin{align*}
	& X_{n+1}^{i, j} = X_n^{i, j} - h  g_i ( X_n^{i, j} ) + \sqrt{ 2 } ( B_{(n+1)h}^{i, j} - B_{nh}^{i, j} ) .
\end{align*}
Define another array $\bW \in \RR^{m \times N}$ through
\begin{align*}
	& W^{i, j} = X_n^{i, j} - h  \bar{V}_i' ( X_n^{i, j} , q_{\bX_n}^{-i} ) + \sqrt{ 2 } ( B_{(n+1)h}^{i, j} - B_{nh}^{i, j} ).
\end{align*}
We have $\bW = \EE ( \bX_{n+1} | \bX_n, \{ \bB_t \}_{ nh \leq t \leq (n+1)h} ,  \{ \bY_t \}_{ nh \leq t \leq (n+1)h}
)$. It can be viewed as the update by Algorithm \ref{gd algo} with full gradient. Then,
\begin{align*}
	\EE | X_{n+1}^{i, j} - Y_{(n+1)h}^{i, \tau_{n,i}(j)} |^2
	& =  \EE [ ( X_{n+1}^{i, j} - W^{i, j} ) + ( W^{i, j} - Y_{(n+1)h}^{i, \tau_{n,i}(j)} ) ]^2 \notag\\
	& = \EE | X_{n+1}^{i, j} - W^{i, j} |^2 + \EE | W^{i, j} - Y_{(n+1)h}^{i, \tau_{n,i}(j)} |^2 .
\end{align*}
From this and \eqref{eqn-W2-frobenius}, we obtain that
\begin{align}
	\EE \cW_2^2 ( q_{\bX_{n+1}} , q_{\bY_{(n+1)h}}  )
	\leq \frac{1}{N} \sum_{ (i,j) \in [m] \times [N] } \Big( \EE | X_{n+1}^{i, j} - W^{i, j} |^2 + \EE | W^{i, j} - Y_{(n+1)h}^{i, \tau_{n,i}(j)} |^2 \Big)
	.
	\label{eqn-lem-recursion-1}
\end{align}
Below, we look into the quantities on the right-hand side. The first term $\EE | X_{n+1}^{i, j} - W^{i, j} |^2 $ is the error incurred by stochastic gradient approximation and hence vanishes as $B \to \infty$. The second term $\EE | W^{i, j} - Y_{(n+1)h}^{i, j} |^2 $ bounds the error caused by finite-difference approximation.

\subsubsection{Step 3: Analysis of $\EE | X_{n+1}^{i, j} - W^{i, j} |^2$}

We will prove that
\begin{align}
	\frac{1}{N} \sum_{ (i,j) \in [m] \times [N] } \EE | X_{n+1}^{i, j} - W^{i, j} |^2 
	\leq  \frac{2h^2 }{B} [ L^2  \EE \cW_2^2 ( q_{\bX_n} , q_{\bY} ) + \EE_{\bx \sim q_*} \| \nabla  V ( \bx ) \|_2^2 ] . 
	\label{eqn-lem-recursion-5}
\end{align}

By definition,
\begin{align*}
	& \frac{1}{N} \sum_{j=1}^{N} \EE | X_{n+1}^{i, j} - W^{i, j} |^2 
	= \frac{h^2}{N} \sum_{j=1}^{N} \EE | g_i ( X_n^{i, j } ) -  \bar{V}_i' ( X_n^{i, j} , q_{\bX_n}^{-i} ) |^2 .
\end{align*}
Let us focus on $i = 1$ for a moment. Conditioned on $\bX_n$, $\{ \bz_n^{-1, b} \}_{b=1}^B$ are i.i.d.~samples from $q_{\bX_n}^{-1}$ and
\[
g_1 ( X_n^{1, j } ) = \frac{1}{B} \sum_{b=1}^B  \partial_1 V ( X_n^{1, j} , \bz_n^{-1, b} ) .
\]
Hence,
\begin{align*}
	& \EE \Big( | g_1 ( X_n^{1, j } ) -  \bar{V}_1' ( X_n^{1, j} , q_{\bX_n}^{-1} ) |^2 \Big| \bX_n \Big)
	= \var [ g_1 ( X_n^{1, j } )
	| \bX_n ] \\
	& = \frac{1}{B} \var [
	\partial_1 V ( X_n^{1, j} , \bz_n^{-1, 1} )
	| \bX_n ] 
	\leq
	\frac{1}{B} \EE \Big\{  [
	\partial_1 V ( X_n^{1, j} , \bz_n^{-1, 1} )
	]^2
	\Big| \bX_n \Big\} .
\end{align*}
Based on the above and the fact that $\bz_n^{:, 1} \sim q_{\bX_n}$,
\begin{align*}
	& \frac{1}{N} \sum_{j=1}^{N} \EE | X_{n+1}^{1, j} - W^{1, j} |^2 
	\leq  \frac{h^2}{BN} \sum_{j=1}^{N} \EE [ 
	\partial_1 V ( X_n^{1, j} , \bz_n^{-1, 1} )
	]^2 
	= \frac{h^2}{B} \EE |
	\partial_1  V ( \bz_n^{:, 1} )
	|^2  .
\end{align*}
As a result,
\begin{align}
	\frac{1}{N} \sum_{ (i,j) \in [m] \times [N] } \EE | X_{n+1}^{i, j} - W^{i, j} |^2 
	\leq  \frac{h^2}{B} \EE \| \nabla  V ( \bz_n^{:, 1} ) \|_2^2
	=  \frac{h^2}{B} \EE_{\bX_n} [ \EE_{\bz \sim q_{\bX_n}} \| \nabla  V ( \bz ) \|_2^2 ] .
	\label{eqn-lem-recursion-4}
\end{align}

We now relate the right-hand side of \eqref{eqn-lem-recursion-4} to $ \EE_{ \by \sim q_{\bY} } \| \nabla  V ( \by ) \|_2^2$. We condition on $( \bX_n , \bY_{nh} )$ and denote by $\gamma$ an $\cW_2$-optimal coupling between $q_{\bX_n}$ and $q_{\bY_{nh}}$ such that
\[
\EE_{ ( \bu, \bv ) \sim \gamma } \| \bu - \bv \|_2^2 = \cW_2^2 ( q_{\bX_n} , q_{\bY_{nh}} ).
\]
Construct random vectors $\bz , \by \in \RR^m$ such that the conditional distribution of $(\bz, \by) $ given $( \bX_n , \bY_{nh} )$ is $\gamma$. Then,
\[
\EE \| \bz - \by \|_2^2
= \EE [ \EE ( \| \bz - \by \|_2^2 | \bX_n , \bY_{nh} ) ] = \EE \cW_2^2 ( q_{\bX_n} , q_{\bY_{nh}} ).
\]
By Assumption \ref{assumption-regularity},
\begin{align*}
	\EE   \| \nabla  V ( \bz ) \|_2^2 
	& = \EE \| [ \nabla  V ( \bz ) - \nabla V(\by) ] + \nabla V (\by) \|_2^2 
	\\& 
	\leq 2 \EE \| \nabla  V ( \bz ) - \nabla V(\by)  \|_2^2 + 2 \EE  \| \nabla  V ( \by ) \|_2^2 \\
	& \leq 2 L^2 \EE \|   \bz - \by  \|_2^2 + 2 \EE  \| \nabla  V ( \by ) \|_2^2 
	\leq 2 L^2  \EE \cW_2^2 ( q_{\bX_n} , q_{\bY_{nh}} ) + 2 \EE  \| \nabla  V ( \by ) \|_2^2   .
\end{align*}
The conditional distribution of $\by$ given $\bY_{nh}$ is $q_{\bY_{nh}}$. Hence, $\by \sim q_*$. Plugging the above estimate into \eqref{eqn-lem-recursion-4}, we get \eqref{eqn-lem-recursion-5}.

\subsubsection{Step 4: Analysis of $ \EE | W^{i, j} - Y_{(n+1)h}^{i, \tau_{n,i}(j)} |^2 $}
We will show that
\begin{align}
	\frac{1}{N} \sum_{ (i,j) \in [m] \times [N] }	\EE | W^{i, j} - Y_{(n+1)h}^{i, \tau_{n,i}(j)}  |^2 
	\leq & ( 1 - \alpha h )  \EE \cW_2^2 ( q_{\bX_n}, q_{\bY_{nh}} )
	\nonumber \\& 
	+ \frac{ 4 h^2 L^2 }{ \alpha }  [ h\EE_{ \bx \sim q_* } \| \nabla V(\bx) \|_2^2   + m] .
	\label{eqn-lem-recursion-3-1}
\end{align}
The error bound consists of a contraction term and an $O(h^2)$ additive term.

To prove the result, we first use \eqref{eqn-Y-new} to get
\begin{align*}
	& Y_{(n+1)h}^{i, \tau_{n,i}(j)} 
	= Y_{nh}^{i, \tau_{n,i}(j)} - \int_{nh}^{(n+1)h}  \bar{V}_i' ( Y_{t}^{i, \tau_{n,i}(j)} , q_*^{-i} )  \rd t + \sqrt{2} ( B_{(n+1)h}^{i, j} - B_{nh}^{i, j} ) \\
	& = [ Y_{nh}^{i, \tau_{n,i}(j)} - h  \bar{V}_i' ( Y_{nh}^{i, \tau_{n,i}(j)} , q_*^{-i} )   ] \\
	& \quad
	- \int_{nh}^{(n+1)h} [   \bar{V}_i' ( Y_{t}^{i, \tau_{n,i}(j)} , q_*^{-i} ) 
	-   \bar{V}_i' ( Y_{nh}^{i, \tau_{n,i}(j)} , q_*^{-i} ) 
	]
	\rd t  + \sqrt{2}  ( B_{(n+1)h}^{i, j} - B_{nh}^{i, j} ) .
\end{align*}
Then,
\begin{align*}
	W^{i, j} - Y_{(n+1)h}^{i, \tau_{n,i}(j)} 
	& = \Big( \underbrace{
		[ X_n^{i, j} - h  \bar{V}_i' ( X_n^{i, j} , q_{\bX_n}^{-i} ) ]
	}_{ U_+^{i, j} }
	- 
	\underbrace{
		[ Y_{nh}^{i, \tau_{n,i}(j)} - h  \bar{V}_i' ( Y_{nh}^{i, \tau_{n,i}(j)} , q_*^{-i} )   ] 
	}_{ Y_+^{i, j} }
	\Big) 
	\\&\quad 
	- 
	\underbrace{
		\int_{nh}^{(n+1)h} [   \bar{V}_i' ( Y_{t}^{i, \tau_{n,i}(j)} , q_*^{-i} ) 
		-   \bar{V}_i' ( Y_{nh}^{i, \tau_{n,i}(j)} , q_*^{-i} ) 
		]
		\rd t 
	}_{ A^{i, j} }
	.
\end{align*}
We have
\begin{align}
	\EE | W^{i, j} - Y_{(n+1)h}^{i, \tau_{n,i}(j)}  |^2
	= \EE | U_+^{i, j} - Y_+^{i, j} |^2 - 2 \EE [
	( U_+^{i, j} - Y_+^{i, j}  ) A^{i, j}
	] + \EE | A^{i, j} |^2 .
	\label{eqn-decomposition}
\end{align}

We make two claims and defer their proofs to \Cref{sec-claim-cross-terms-proof,sec-claim-contraction-proof}, respectively.

\begin{claim}\label{claim-cross-terms}
	Choose any $\varepsilon > 0$. Under Assumption \ref{assumption-regularity}, we have
	\begin{align*}
		& -2 \EE [
		( U_+^{i, j} - Y_+^{i, j}  ) A^{i, j} 
		]  + \EE | A^{i, j} |^2 
		\notag \\& 
		\leq 
		\varepsilon h  \EE | U_+^{i, j} - Y_+^{i, j}  |^2 +
		2 (\varepsilon^{-1} + h)
		L^2 h^2 [ h \EE_{ \bx \sim q_* } | \partial_i V(\bx) |^2 + 1 ].
	\end{align*}
\end{claim}

\begin{claim}\label{claim-contraction}
	Let Assumption \ref{assumption-regularity} hold and $0 < h < 2 / (\alpha + L)$. For any $j \in [N]$,
	\[
	\frac{1}{N} \sum_{ (i,j) \in [m] \times [N] }
	\EE | U_+^{i,j} - Y_+^{i, j}  |^2  
	\leq ( 1 - \alpha h )^2  \EE \cW_2^2 ( q_{\bX_n}, q_{\bY_{nh}} ) 
	\]
\end{claim}

By \eqref{eqn-decomposition} and Claim \ref{claim-cross-terms},
\[
\EE | W^{i, j} - Y_{(n+1)h}^{i, \tau_{n,i}(j)}  |^2
\leq ( 1 + \varepsilon h ) \EE | U_+^{i, j} - Y_+^{i, j}  |^2 
+ 2 (\varepsilon^{-1} + h)
L^2 h^2 [ h \EE_{ \bx \sim q_* } | \partial_i V(\bx) |^2 + 1 ] .
\]
Based on Claim \ref{claim-contraction},
\begin{align*}
	& \frac{1}{N} \sum_{ (i,j) \in [m] \times [N] }	\EE | W^{i, j} - Y_{(n+1)h}^{i, \tau_{n,i}(j)}  |^2 \notag\\
	\leq &   \frac{
		1 + \varepsilon h	
	}{N} \sum_{ (i,j) \in [m] \times [N] }	\EE | U_+^{i, j} - Y_+^{i, j}  |^2
	+ 2 ( \varepsilon^{-1}  + h )  L^2 h^2 [ h\EE_{ \bx \sim q_* } \| \nabla V(\bx) \|_2^2   + m] \\
	\leq& ( 1 + \varepsilon h )  ( 1 - \alpha h )^2  \EE \cW_2^2 ( q_{\bX_n}, q_{\bY_{nh}} ) 
	+ 2 ( \varepsilon^{-1}  + h )  L^2 h^2 [ h \EE_{ \bx \sim q_* } \| \nabla V(\bx) \|_2^2   + m].
\end{align*}
Next, take $\varepsilon = \alpha$. Since $h < 2 /(\alpha + L) \leq 1 / \alpha$, we have $	\varepsilon^{-1}  + h \leq 2 / \alpha$. Then, the result \eqref{eqn-lem-recursion-3-1} becomes obvious.

\subsubsection{Step 5: Proving the recursive inequality}
By \eqref{eqn-lem-recursion-1}, \eqref{eqn-lem-recursion-5} and \eqref{eqn-lem-recursion-3-1},
\begin{align*}
	\EE \cW_2^2 ( q_{\bX_{n+1}}, q_{\bY_{(n+1)h}} ) 
	&	\leq  \bigg( 1 - \alpha h +  \frac{2L^2 h^2}{B}   \bigg)    \EE \cW_2^2 ( q_{\bX_n}, q_{\bY_{nh}} ) 
	\notag	\\
	&\quad  +	\bigg(
	\frac{2h^2}{B} + \frac{4h^3 L^2}{\alpha} 
	\bigg) \EE_{ \bx \sim q_* } \| \nabla V(\bx) \|_2^2 
	+  \frac{4 m h^2 L^2}{\alpha}  
	.
\end{align*}
The assumption $h < \frac{B \alpha}{4 L^2}$ implies $2L^2 h^2 / B < \alpha h / 2 $. By \Cref{lem-grad-square}, we have $ \EE_{ \bx \sim q_* } \| \nabla V(\bx) \|_2^2 \leq m L^2 / \alpha$. Combining the above estimates finishes the proof.

\subsection{Proof of Claim \ref{claim-cross-terms}}\label{sec-claim-cross-terms-proof}
Choose any $\varepsilon > 0$. We have
\begin{align*}
	& -2  [
	( U_+^{i, j} - Y_+^{i, j}  ) A^{i, j} 
	]  
	\leq   
	\varepsilon h  | U_+^{i, j} - Y_+^{i, j}  |^2 + \frac{
		|A^{i, j} |^2
	}{\varepsilon h} 
\end{align*}
and thus,
\begin{align}
	-2 \EE [
	( U_+^{i, j} - Y_+^{i, j}  ) A^{i, j} 
	]  + \EE | A^{i, j} |^2 
	\leq   
	\varepsilon h  \EE | U_+^{i, j} - Y_+^{i, j}  |^2 + (\varepsilon^{-1} + h) h^{-1} \EE | A^{i, j} |^2 .
	\label{eqn-claim-new-1}
\end{align}
By the definition of $A^{i, j}$ and \Cref{lem-grad-square}, 
\begin{align*}
	|A^{i, \tau_{n,i}^{-1}(j)}| \leq 
	\int_{nh}^{(n+1)h} 
	| \bar{V}_i' ( Y_{t}^{i, j} , q_*^{-i} ) 
	-   \bar{V}_i' ( Y_{nh}^{i, j} , q_*^{-i} ) 
	|
	\rd t 
	\leq \int_{nh}^{(n+1)h} 
	L | Y_{t}^{i, j}  - Y_{nh}^{i, j} | \rd t.
\end{align*}
Then,
\begin{align}
	\EE |A^{i, \tau_{n,i}^{-1}(j)}|^2 
	& \leq
	\EE \bigg(
	\int_{0}^{h} 
	L | Y_{nh+t}^{i, j}  - Y_{nh}^{i, j} | \rd t
	\bigg)^2 
	= L^2 h^2 \EE \bigg(
	\frac{1}{h}
	\int_{0}^{h} 
	| Y_{nh+t}^{i, j}  - Y_{nh}^{i, j} | \rd t
	\bigg)^2 \notag \\
	& \leq L^2 h^2 \EE \bigg(
	\frac{1}{h} \int_{0}^{h} 
	| Y_{nh+t}^{i, j}  - Y_{nh}^{i, j} |^2 \rd t
	\bigg)
	= L^2 h  
	\int_{0}^{h} 
	\EE | Y_{nh+t}^{i, j}  - Y_{nh}^{i, j} |^2 \rd t.
	\label{eqn-claim-new-2}
\end{align}

By \eqref{eqn-Y-new}, we have 
\begin{align}
	| Y_{nh+t}^{i, j}  - Y_{nh}^{i, j} |^2 
	& = \bigg(  - \int_{0}^{t} 
	\bar{V}_i' ( Y_{nh+s}^{i, j} , q_*^{-i} ) \rd s + \sqrt{2} 
	(
	B_{nh+t}^{i, \tau_{n, i}^{-1} (j) }
	- B_{nh}^{i, \tau_{n, i}^{-1} (j) }
	)
	\bigg)^2 
	\notag \\& 
	\leq 2 \bigg[ \bigg(  \int_{0}^{t} 
	\bar{V}_i' ( Y_{nh+s}^{i, j} , q_*^{-i} ) 
	\rd s \bigg)^2 +  [\sqrt{2} 
	(
	B_{nh+t}^{i, \tau_{n, i}^{-1} (j) }
	- B_{nh}^{i, \tau_{n, i}^{-1} (j) }
	)
	]^2 \bigg] .
	\label{eqn-claim-new-3}
\end{align}

Note that
\begin{align}
	\bigg(
	\int_{0}^{t} 
	\bar{V}_i' ( Y_{nh+s}^{i, j} , q_*^{-i} ) 
	\rd s 
	\bigg)^2
	=  t^2 \bigg( \frac{1}{t} \int_{0}^{t} 
	\bar{V}_i' ( Y_{nh+s}^{i, j} , q_*^{-i} ) 
	\rd s \bigg)^2 
	\leq t \int_{0}^{t} | 
	\bar{V}_i' ( Y_{nh+s}^{i, j} , q_*^{-i} ) 
	|^2 \rd s.
	\label{eqn-claim-new-4}
\end{align}
Since $Y_{nh+s}^{:, j} \sim q_*$, we have
\begin{align}
	\EE |  \bar{V}_i'  ( Y_{nh+s}^{i, j}, q_*^{-i} ) |^2
	&	= \EE \Big|  \EE_{\by \sim q_*^{-i}} [ \partial_i V ( Y_{nh+s}^{i, j}, \by ) ] \Big|^2 
	= \EE \Big|  \EE [  \partial_i V ( \bY_{nh+s}^{:, j} ) | Y_s^{i, j} ] \Big|^2
	\notag	\\& 
	\leq \EE | \partial_i V  ( \bY_{nh+s}^{:, j} ) |^2
	= \EE_{ \bx \sim q_* } | \partial_i V(\bx) |^2.
	\label{eqn-claim-new-5}
\end{align}

By \eqref{eqn-claim-new-3}, \eqref{eqn-claim-new-4} and \eqref{eqn-claim-new-5},
\begin{align*}
	\EE	| Y_{nh+t}^{i, j}  - Y_{nh}^{i, j} |^2 
	&	\leq 2 t \int_{0}^{t} \EE | \bar{V}_i'  ( Y_{nh+s}^{i, j}, q_*^{-i} ) |^2 \rd s +  4 \EE |
	B_{nh+t}^{i, \tau_{n, i}^{-1} (j) }
	- B_{nh}^{i, \tau_{n, i}^{-1} (j) }
	|^2 \\
	& \leq 2t^2 \EE_{ \bx \sim q_* } | \partial_i V(\bx) |^2 + 4 t.
\end{align*}
Combining this with \eqref{eqn-claim-new-2}, we obtain that
\[
\EE |A^{i, \tau_{n,i}^{-1}(j)}|^2 
\leq 
L^2 h  
\int_{0}^{h} 
\Big(
2t^2 \EE_{ \bx \sim q_* } | \partial_i V(\bx) |^2 + 4 t
\Big)
\rd t
= \frac{2L^2 h^4}{3} \EE_{ \bx \sim q_* } | \partial_i V(\bx) |^2 + 2L^2 h^3.
\]
Plugging this into \eqref{eqn-claim-new-1} yields the claimed result.

\subsection{Proof of Claim \ref{claim-contraction}}\label{sec-claim-contraction-proof}

Let $\{ \sigma_i \}_{i=1}^m$ be i.i.d.~permutations of $[N]$ that are independent of $(\bX_n , \bY_{nh} )$. Denote by $\cF^i$ the $\sigma$-field generated by $(\bX_n^{i,:}, \bY_{nh}^{i, :}, \sigma_i)$. Define two arrays $\bU, \bZ \in \RR^{m \times N}$ through $U^{i,j} = X_n^{i, \sigma_i(j)}$ and $Z^{i,j} = Y_{nh}^{i, \sigma_i\circ \tau_{n,i} (j)}$. We have
\begin{align}
	\sum_{j=1}^N
	| U_+^{i,j} - Y_+^{i, j}  |^2  
	& = \sum_{j=1}^N 
	\{
	[
	X_n^{i, j} - h  \bar{V}_i' ( X_n^{i, j} , q_{\bX_n}^{-i} ) 
	] - [
	Y_{nh}^{i, \tau_{n,i}(j)} - h  \bar{V}_i' ( Y_{nh}^{i, \tau_{n,i}(j)} , q_{*}^{-i} ) 
	]
	\}^2 \notag\\
	& =  \sum_{j=1}^N 
	\{
	[
	U^{i, j} - h  \bar{V}_i' ( U^{i, j} , q_{\bX_n}^{-i} ) 
	] - [
	Z^{i,j} - h  \bar{V}_i' ( Z^{i,j} , q_{*}^{-i} ) 
	]
	\}^2 \notag\\
	& =  \sum_{j=1}^N 
	\{
	[
	U^{i, j} - h  \bar{V}_i' ( U^{i, j} , q_{\bU_n}^{-i} ) 
	] - [
	Z^{i,j} - h  \bar{V}_i' ( Z^{i,j} , q_{*}^{-i} ) 
	]
	\}^2,
	\label{eqn-contraction-1}
\end{align}
where the last equality follows from the fact that $q_{\bX_n} = q_{\bU}$.

By definition, $\bU^{i, :}$ and $\bZ^{i, :}$ are $\cF^i$-measurable. Conditioned on $\cF^i$, we have $\bZ_{nh}^{-i, j} \sim q_*^{-i}$ and $\bU^{-i, j} \sim q_{ \bU }^{-i}$. 
As a result,
\begin{align*}
	& \EE [ \partial_i V ( \bZ^{ :,j } ) | \cF^i ]
	= \EE [ \partial_i V ( Z^{ i,j }, \bZ^{ -i,j } ) | \cF^i ]
	= \EE_{ \bz \sim q_*^{-i} } \partial_i V ( Z^{ i,j }, \bz ) 
	=  \bar{V}_i' ( Z^{ i,j } , q_*^{-i} ) ,\\
	& Z^{i, j} - h  \bar{V}_i' ( Z^{i, j} , q_{*}^{-i} ) =  \EE [ Z^{i, j} - h \partial_i V ( \bZ^{ :,j } ) | \cF^i ] .
\end{align*}
Similarly, we have $\EE [ \partial_i V ( \bU^{ :,j } ) | \cF^i ] =  \bar{V}_i' ( U^{ i,j } , q_{ \bU }^{-i} ) $ and
\[
U^{i, j} - h  \bar{V}_i' ( U^{i, j} , q_{\bU}^{-i} ) =  \EE [ U^{i, j} - h \partial_i V ( \bU^{ :,j } ) | \cF^i ] .
\]
Consequently,
\begin{align*}
	& \EE \{
	[
	U^{i, j} - h  \bar{V}_i' ( U^{i, j} , q_{\bU_n}^{-i} ) 
	] - [
	Z^{i,j} - h  \bar{V}_i' ( Z^{i,j} , q_{*}^{-i} ) 
	]
	\}^2
	\\
	& = \EE \Big| \EE \Big( [U^{i,  j  } - h  \partial_i V ( \bU^{ :,  j  } ) ]
	- [Z^{i, j} - h \partial_i V ( \bZ^{ :,j } )]
	\Big| \cF^i \Big) \Big|^2 \\
	& \leq  \EE \Big|  [U^{i, j } - h  \partial_i V ( \bU^{ :, j } ) ]
	- [Z^{i, j} - h \partial_i V ( \bZ^{ :,j } )]
	\Big|^2 .
\end{align*}
Define $\bphi:~ \RR^{mr} \to \RR^{mr}$, $\bx \mapsto \bx - h \nabla V (\bx)$. Then,
\begin{align*}
	& \sum_{i=1}^{m} \EE \{
	[
	U^{i, j} - h  \bar{V}_i' ( U^{i, j} , q_{\bU_n}^{-i} ) 
	] - [
	Z^{i,j} - h  \bar{V}_i' ( Z^{i,j} , q_{*}^{-i} ) 
	]
	\}^2 \\
	& \leq 
	\EE \| \bphi ( \bU^{:, j }  ) - \bphi ( \bZ^{:, j }  ) \|_2^2 
	\leq ( 1 - \alpha h )^2  \EE \|   \bU^{:, j }  - \bZ^{:, j }  \|_2^2 ,
\end{align*}
where we used \Cref{lem-cvx} and the assumption $0 < h < 2 / (\alpha + L)$. Based on the above, we use \eqref{eqn-contraction-1} to get
\begin{align*}
	\frac{1}{N} \sum_{ (i,j) \in [m] \times [N] }
	\EE | U_+^{i,j} - Y_+^{i, j}  |^2  
	& \leq \frac{ ( 1 - \alpha h )^2 }{N}   \sum_{ (i,j) \in [m] \times [N] }\EE |  U^{i, j }  - Z^{i, j }  |^2 \\
	& = \frac{ ( 1 - \alpha h )^2 }{N}   \sum_{ (i,j) \in [m] \times [N] }\EE |  X_{n}^{i, j }  - Y_{nh}^{i, \tau_{n,i}(j) }  |^2.
\end{align*}
The proof is finished by \eqref{eqn-W2-frobenius-0}.

\section{Technical lemmas}\label{sec: technical lems}

\begin{lemma}\label{lem-W2-triangle}
	Let $\mu$, $\nu$ and $\rho$ be random probability distributions on $\RR^m$. Then,
	\begin{align*}
		\Big|
		\sqrt{ \EE \cW_2^2 (  \mu, \nu ) } - \sqrt{ \EE \cW_2^2 ( \mu , \rho ) }
		\Big| 
		\leq 	\sqrt{ \EE \cW_2^2 ( \nu , \rho  ) } .
	\end{align*}
\end{lemma}

\begin{proof}[\bf Proof of \Cref{lem-W2-triangle}]
	By the triangle inequality of $\cW_2$, we have $| \cW_2 ( \mu, \nu ) - \cW_2 ( \mu, \rho ) | \leq \cW_2 ( \nu, \rho )$. By Minkowski's inequality,
	\begin{align*}
		\sqrt{ \EE \cW_2^2 ( \nu , \rho  ) }& \geq 
		\sqrt{ \EE | \cW_2 ( \mu, \nu ) - \cW_2 ( \mu, \rho ) |^2 } \geq \Big|
		\sqrt{ \EE \cW_2^2 (  \mu, \nu ) } - \sqrt{ \EE \cW_2^2 ( \mu , \rho ) }
		\Big| \, .
	\end{align*}
\end{proof}

\begin{lemma}[Additivity of $\cW_2$ for product measures]\label{lem-additivity-W2}
	Let $\{ d_i \}_{i=1}^m$ be positive integers, and $p_i, q_i \in \cP ( \RR^{d_i} )$ for $i \in [m]$. We have
	\[
	\cW_2^2  (
	p_1 \otimes \cdots \otimes p_m
	,~ q_1 \otimes \cdots \otimes q_m
	) = \sum_{i=1}^m \cW_2^2 ( p_i, q_i  ).
	\]
\end{lemma}

\begin{proof}[\bf Proof of \Cref{lem-additivity-W2}]
	Let $p = p_1 \otimes \cdots \otimes p_m$ and $q = q_1 \otimes \cdots \otimes q_m$. For $i \in [m]$ and $\bx \in \RR^{d_1 + \cdots + d_m}$, denote by $\bx_i$ the $(1+ \sum_{j=1}^{i-1} d_j)$-th through $\sum_{j=1}^{i} d_j$-th coordinates of $\bx$; let $\gamma_i$ be a $\cW_2$-optimal coupling between $p_i$ and $q_i$. Then, $\gamma_1 \otimes \cdots \otimes \gamma_m$ is a coupling between $p$ and $q$. For $(\bX, \bY) \sim \gamma$, we have
	\[
	\cW_2^2 
	(p,q) 
	\leq \EE \| \bX - \bY \|_2^2
	= \sum_{i=1}^{m} \EE \| \bX_i - \bY_i \|^2 = \sum_{i=1}^{m} \cW_2^2 ( p_i, q_i ).
	\]
	To prove the opposite direction, define $\gamma$ as a $\cW_2$-optimal coupling between $p$ and $q$. If we draw $(\bX, \bY) \sim \gamma$, then $\bX_i \sim p_i$, $\bY_i \sim q_i$, and
	\[
	\cW_2^2 (p,q) 
	= \EE \| \bX - \bY \|_2^2
	=  \sum_{i=1}^{m} \EE \| \bX_i - \bY_i \|^2 \geq \sum_{i=1}^{m} \cW_2^2 ( p_i, q_i ).
	\]
\end{proof}

\begin{lemma}[Lemma 1 in \cite{Dal17}]\label{lem-cvx}
	Let Assumption \ref{assumption-regularity} hold and $0 < h < 2 / (\alpha + L)$. Define $\bphi (\bx) = \bx - h \nabla V (\bx)$, $\forall \bx \in \RR^{m}$. We have
	\[
	\| \bphi (\bx) - \bphi (\by) \|_2 \leq ( 1 - \alpha h ) \cdot \| \bx - \by \|_2 , \qquad \forall \bx, \by \in \RR^{m}.
	\]
\end{lemma}

\begin{lemma}\label{lem-grad-square}
	Let Assumption \ref{assumption-regularity} hold. 
	\begin{itemize}
		\item For any $i \in [m]$, $\mu \in \cP ( \RR^{m-1} )$ and $x \in \RR$, we have $\alpha \leq  \bar{V}_i'' ( x , \mu  ) \leq L $. 
		\item For $\bx \sim q_*$, we have $\EE [  \nabla V (\bx) ] = \boldsymbol{0}$, $\EE  \| \nabla V (\bx) \|_2^2 \leq m L^2 \alpha^{-1}$ and $\cov   (\bx) \preceq  \alpha^{-1} \bI_m$.
	\end{itemize}
\end{lemma}

\begin{proof}[\bf Proof of \Cref{lem-grad-square}]
	The proof of $\alpha \leq  \bar{V}_i'' ( x , \mu  ) \leq L $ is immediate and thus omitted. We now derive some useful results from that. By the fixed-point characterization (\Cref{lem-fixed-point}), the density function of $q_*^i$ is proportional to $e^{-V_i(\cdot)}$, where $V_i(\cdot) =  \bar{V}_i (\cdot , q_*^{-i})$. We have $\alpha \leq V_i''(\cdot) \leq L$. Hence, $q_*^i$ is strongly log-concave.
	
	We now work on the second bullet point in the lemma. By definition,
	\begin{align*}
		\EE_{ \bx \sim q_*  }  \partial_1 V (\bx)
		& = \EE_{ x_1 \sim q_*^1  } [
		\EE_{ \bx^{-1} \sim q_*^{-1} } \partial_1 V ( x_1, \bx^{-1} )
		]
		= \EE_{ x_1 \sim q_*^1  } \bar{V}_1' (x_1 , q_*^{-1}) \\
		&
		= \EE_{ x_1 \sim q_*^1  } V_1'(x_1)
		\propto \int_{\RR} V_1'(x) e^{-V_1(x)} \rd x
		= - \int_{\RR} \rd e^{-V_1(x)} = 0.
	\end{align*}
	The last inequality follows from the fact that $V_1'' \geq \alpha$ and thus $\lim_{|x| \to \infty} V_1(x) = + \infty$. Hence, $\EE_{ \bx \sim q_*  } [ \nabla V (\bx) ] = \boldsymbol{0}$. 
	
	Consequently, $\EE_{ \bx \sim q_*  } | \partial_1 V (\bx) |^2
	= \var_{ \bx \sim q_*  } [ \partial_1 V (\bx) ]$. Since $q_*$ is a product distribution with strongly log-concave marginals, the Bakry-\'{E}mery principle \citep{BEm06} implies the Poincar\'{e} inequality:
	\begin{align}
		\var_{ \bx \sim q_*  } [ f (\bx) ] \leq \alpha^{-1} \EE_{ \bx \sim q_*  } \| \nabla f (\bx) \|_2^2 , \qquad \forall f \in C^1 (\RR^m).
		\label{eqn-poincare}
	\end{align}
	Hence,
	\[
	\var_{ \bx \sim q_*  } [ \partial_1 V (\bx) ] \leq \alpha^{-1} \EE_{ \bx \sim q_*  } \| \nabla [  \partial_1 V (\bx) ] \|_2^2 
	= \alpha^{-1} \sum_{j=1}^{m} \EE_{ \bx \sim q_* } | \partial^2_{j1} V(\bx) |^2 .
	\]
	By Assumption \ref{assumption-regularity},
	\begin{align*}
		\EE_{ \bx \sim q_*  } \| \nabla V (\bx) \|_2^2 
		& = \alpha^{-1} \sum_{i, j \in [m]} \EE_{ \bx \sim q_* } | \partial^2_{ij} V(\bx) |^2 
		= \alpha^{-1}  \EE_{ \bx \sim q_* }  \| \nabla^2 V(\bx) \|_{\mathrm{F}}^2 \\
		& \leq \alpha^{-1} \EE_{ \bx \sim q_* }  [ m \| \nabla^2 V(\bx) \|_{2}^2 ] \leq m L^2 / \alpha.
	\end{align*}
	
	Next, choose any unit vector $\bu \in \RR^m$. By \eqref{eqn-poincare}, we have
	\begin{align}
		\var_{ \bx \sim q_*  } ( \langle \bu , \bx \rangle )  \leq \alpha^{-1} \EE_{ \bx \sim q_*  } \| \bu \|_2^2 = \alpha^{-1}.
	\end{align}
	Therefore, $\cov_{ \bx \sim q_*^i  }  (\bx) \leq  \alpha^{-1} \bI_m$. 
\end{proof}

\begin{lemma}[Corollary 6.12 in \cite{BLe19}]\label{lem-W2-concentration}
	Let $\mu \in \cP ( \RR )$ be log-concave with variance $\sigma^2$, and $\{ X_j \}_{j=1}^N$ be i.i.d.~samples from $\mu$. There exists a universal constant $c$ such that
	\[
	\EE \cW_2^2 \bigg(
	\frac{1}{N} \sum_{j=1}^{N} \delta_{ X_i } , \mu
	\bigg) \leq \frac{c \sigma^2 \log N }{N} , \qquad \forall N \geq 2 .
	\]
\end{lemma}

\Cref{lem-additivity-W2,lem-W2-concentration} yield the following corollary.

\begin{corollary}\label{cor-W2-concentration}
	Let Assumption \ref{assumption-regularity} hold, $N \geq 2$, and $\bX \in \RR^{m \times N}$ be a random array. Suppose that the columns of $\bX$ are drawn independently from a product distribution $q$, whose marginals are log-concave with variance bounded by $\sigma^2$. There exists a universal constant $c$ such that
	\[
	\EE \cW_2^2 ( q_{\bX}, q ) \leq \frac{c m \sigma^2 \log N }{ N} ,
	\]
	where $q_{\bX}$ is constructed according to \Cref{defn-particles}.
\end{corollary}

\section{Discussions}\label{sec: Discussions}

We introduced PArticle VI (PAVI), a particle algorithm for mean-field variational inference (MFVI). The algorithm is based on formulating the relevant optimization problem in terms of Fokker-Planck equations and corresponding diffusion processes. Our theoretical analysis provides non-asymptotic $\cW_2$ error bounds for the algorithm. Several future directions are worth pursuing. For instance, one may consider relaxing the regularity assumptions in \Cref{thm-main} and adapting PAVI to the setting of parametric MFVI.
Another interesting direction is to develop accelerated versions of PAVI using the recipe in the recent works \citep{wang2022accelerated,chen2023accelerating}.

\section*{Acknowledgement}
The authors thank Daniel Lacker for helpful discussions. Du's research is supported in part by National Science Foundation grants DMS-2309245 and DMS-1937254, and the Department of Energy grant DE-SC0025347. Wang's research is supported by the National Science Foundation grant DMS-2210907.


{
\bibliographystyle{ims}
\bibliography{references_abbr}

@book {bishop2006pattern,
    AUTHOR = {Bishop, Christopher M.},
     TITLE = {Pattern recognition and machine learning},
    SERIES = {Information Science and Statistics},
 PUBLISHER = {Springer, New York},
      YEAR = {2006},
     PAGES = {xx+738},
      ISBN = {978-0387-31073-2; 0-387-31073-8},
   MRCLASS = {62-01 (62H30 62J05 62M45 68-01 68T05 68T10)},
  MRNUMBER = {2247587},
}

@article{yao2022mean,
  title={Mean-field variational inference via {Wasserstein} gradient flow},
  author={Yao, Rentian and Yang, Yun},
  journal={arXiv preprint arXiv:2207.08074},
  year={2022}
}

@article{lambert2022variational,
  title={Variational inference via {Wasserstein} gradient flows},
  author={Lambert, Marc and Chewi, Sinho and Bach, Francis and Bonnabel, Silv{\`e}re and Rigollet, Philippe},
  journal={Advances in Neural Information Processing Systems},
  volume={35},
  pages={14434--14447},
  year={2022}
}

@article{garcia2020bayesian,
  title={The {B}ayesian Update: Variational Formulations and Gradient Flows},
  author={Trillos, Nicolas Garcia and Sanz-Alonso, Daniel},
  journal={Bayesian Anal.},
  volume={15},
  number={1},
  pages={29--56},
  year={2020}
}

@book{ambrosio2005gradient,
  title={Gradient flows: in metric spaces and in the space of probability measures},
  author={Ambrosio, Luigi and Gigli, Nicola and Savar{\'e}, Giuseppe},
  year={2005},
  publisher={Springer Science \& Business Media}
}

@article{santambrogio2017euclidean,
  title={$\{$Euclidean, metric, and Wasserstein$\}$ gradient flows: an overview},
  author={Santambrogio, Filippo},
  journal={Bull. Math. Sci.},
  volume={7},
  pages={87--154},
  year={2017},
  publisher={Springer}
}

@inproceedings{gershman2012nonparametric,
  title={Nonparametric variational inference},
  author={Gershman, Samuel J and Hoffman, Matthew D and Blei, David M},
  booktitle={Proceedings of the 29th International Coference on International Conference on Machine Learning},
  pages={235--242},
  year={2012}
}

@article{durmus2019analysis,
  title={Analysis of {L}angevin {M}onte {C}arlo via convex optimization},
  author={Durmus, Alain and Majewski, Szymon and Miasojedow, B{\l}a{\.z}ej},
  journal={J. Mach. Learn. Res},
  volume={20},
  number={73},
  pages={1--46},
  year={2019}
}

@article{dalalyan2017theoretical,
  title={Theoretical guarantees for approximate sampling from smooth and log-concave densities},
  author={Dalalyan, Arnak S},
  journal={J. R. Stat. Soc. Ser. B. Stat. Methodol.},
  volume={79},
  number={3},
  pages={651--676},
  year={2017},
  publisher={Oxford University Press}
}

@inproceedings{cheng2018convergence,
  title={Convergence of {L}angevin {MCMC} in {KL}-divergence},
  author={Cheng, Xiang and Bartlett, Peter},
  booktitle={Algorithmic Learning Theory},
  pages={186--211},
  year={2018},
  organization={PMLR}
}

@article{JCP23,
	title={Algorithms for mean-field variational inference via polyhedral optimization in the {W}asserstein space},
	author={Jiang, Yiheng and Chewi, Sinho and Pooladian, Aram-Alexandre},
        journal={Foundations of Computational Mathematics},
        pages={1--52},
        year={2025}
}

@incollection{BEm06,
	title={Diffusions hypercontractives},
	author={Bakry, Dominique and {\'E}mery, Michel},
	booktitle={S{\'e}minaire de Probabilit{\'e}s XIX 1983/84: Proceedings},
	pages={177--206},
	year={2006},
	publisher={Springer}
}

@book{BLe19,
	title={One-dimensional empirical measures, order statistics, and {K}antorovich transport distances},
	author={Bobkov, Sergey and Ledoux, Michel},
	volume={261},
	number={1259},
	year={2019},
	publisher={American Mathematical Society}
}

@article{FGu15,
	title={On the rate of convergence in {W}asserstein distance of the empirical measure},
	author={Fournier, Nicolas and Guillin, Arnaud},
	journal={Probab. Theory Related Fields},
	volume={162},
	number={3-4},
	pages={707--738},
	year={2015},
	publisher={Springer}
}

@inproceedings{Dal17,
	title={Further and stronger analogy between sampling and optimization: {L}angevin {M}onte {C}arlo and gradient descent},
	author={Dalalyan, Arnak},
	booktitle={Conference on Learning Theory},
	pages={678--689},
	year={2017},
	organization={PMLR}
}

@article{Lac23,
	title={Independent projections of diffusions: Gradient flows for variational inference and optimal mean field approximations},
	author={Lacker, Daniel},
	journal={arXiv preprint arXiv:2309.13332},
	year={2023}
}

@book{Vil09,
	title={Optimal transport: old and new},
	author={Villani, C{\'e}dric},
	volume={338},
	year={2009},
	publisher={Springer}
}

@inproceedings{diao2023forward,
  title={Forward-backward {Gaussian} variational inference via {JKO} in the {Bures-Wasserstein} Space},
  author={Diao, Michael Ziyang and Balasubramanian, Krishna and Chewi, Sinho and Salim, Adil},
  booktitle={International Conference on Machine Learning},
  pages={7960--7991},
  year={2023},
  organization={PMLR}
}

@article{blei2017variational,
  title={Variational inference: A review for statisticians},
  author={Blei, David M and Kucukelbir, Alp and McAuliffe, Jon D},
  journal={J. Amer. Statist. Assoc.},
  volume={112},
  number={518},
  pages={859--877},
  year={2017},
  publisher={Taylor \& Francis}
}

@article{wainwright2008graphical,
  title={Graphical models, exponential families, and variational inference},
  author={Wainwright, Martin J and Jordan, Michael I and others},
  journal={Foundations and Trends{\textregistered} in Machine Learning},
  volume={1},
  number={1--2},
  pages={1--305},
  year={2008},
  publisher={Now Publishers, Inc.}
}

@article{jordan1999introduction,
  title={An introduction to variational methods for graphical models},
  author={Jordan, Michael I and Ghahramani, Zoubin and Jaakkola, Tommi S and Saul, Lawrence K},
  journal={Mach. Learn.},
  volume={37},
  pages={183--233},
  year={1999},
  publisher={Springer}
}

@article{jordan1998variational,
  title={The variational formulation of the {F}okker--{P}lanck equation},
  author={Jordan, Richard and Kinderlehrer, David and Otto, Felix},
  journal={SIAM J. Math. Anal.},
  volume={29},
  number={1},
  pages={1--17},
  year={1998},
  publisher={SIAM}
}

@article{mukherjee2022variational,
  title={Variational inference in high-dimensional linear regression},
  author={Mukherjee, Sumit and Sen, Subhabrata},
  journal={J. Mach. Learn. Res},
  volume={23},
  number={304},
  pages={1--56},
  year={2022}
}

@article{zhang2020convergence,
  title={Convergence rates of variational posterior distributions},
  author={Zhang, Fengshuo and Gao, Chao},
  journal={Ann. Statist.},
  volume={48},
  number={4},
  pages={2180--2207},
  year={2020},
  publisher={JSTOR}
}

@article{lacker2024mean,
  title={Mean field approximations via log-concavity},
  author={Lacker, Daniel and Mukherjee, Sumit and Yeung, Lane Chun},
  journal={Int. Math. Res. Not.},
  volume={2024},
  number={7},
  pages={6008--6042},
  year={2024},
  publisher={Oxford University Press}
}

@article{yan2024learning,
  title={Learning {G}aussian mixtures using the {Wasserstein--Fisher--Rao} gradient flow},
  author={Yan, Yuling and Wang, Kaizheng and Rigollet, Philippe},
  journal={Ann. Statist.},
  volume={52},
  number={4},
  pages={1774--1795},
  year={2024},
  publisher={Institute of Mathematical Statistics}
}

@article{liu2016stein,
  title={Stein variational gradient descent: A general purpose {Bayesian} inference algorithm},
  author={Liu, Qiang and Wang, Dilin},
  journal={Advances in neural information processing systems},
  volume={29},
  year={2016}
}

@article{korba2020non,
  title={A non-asymptotic analysis for {Stein} variational gradient descent},
  author={Korba, Anna and Salim, Adil and Arbel, Michael and Luise, Giulia and Gretton, Arthur},
  journal={Advances in Neural Information Processing Systems},
  volume={33},
  pages={4672--4682},
  year={2020}
}

@article{liu2024towards,
  title={Towards understanding the dynamics of {G}aussian--{S}tein variational gradient descent},
  author={Liu, Tianle and Ghosal, Promit and Balasubramanian, Krishnakumar and Pillai, Natesh},
  journal={Advances in Neural Information Processing Systems},
  volume={36},
  year={2024}
}

@article{gallouet2017jko,
  title={A {JKO splitting scheme for Kantorovich--Fisher--Rao} gradient flows},
  author={Gallou{\"e}t, Thomas O and Monsaingeon, Leonard},
  journal={SIAM J. Math. Anal.},
  volume={49},
  number={2},
  pages={1100--1130},
  year={2017},
  publisher={SIAM}
}

@article{lu2023birth,
  title={Birth--death dynamics for sampling: global convergence, approximations and their asymptotics},
  author={Lu, Yulong and Slep{\v{c}}ev, Dejan and Wang, Lihan},
  journal={Nonlinearity},
  volume={36},
  number={11},
  pages={5731},
  year={2023},
  publisher={IOP Publishing}
}

@article{chizat2018global,
  title={On the global convergence of gradient descent for over-parameterized models using optimal transport},
  author={Chizat, Lenaic and Bach, Francis},
  journal={Advances in neural information processing systems},
  volume={31},
  year={2018}
}

@article{chewi2024,
	title={Log-concave sampling},
	author={Chewi, Sinho},
	journal={Book draft available at https://chewisinho.github.io},
	year={2023}
}

@inproceedings{kook2024sampling,
  title={Sampling from the mean-field stationary distribution},
  author={Kook, Yunbum and Zhang, Matthew S and Chewi, Sinho and Erdogdu, Murat A and Li, Mufan Bill},
  booktitle={The Thirty Seventh Annual Conference on Learning Theory},
  pages={3099--3136},
  year={2024},
  organization={PMLR}
}

@article{bures1969extension,
  title={An extension of {K}akutani's theorem on infinite product measures to the tensor product of semifinite $w^{*}$-algebras},
  author={Bures, Donald},
  journal={Trans. Amer. Math. Soc.},
  volume={135},
  pages={199--212},
  year={1969},
  publisher={JSTOR}
}

@article{huix2024theoretical,
  title={Theoretical Guarantees for Variational Inference with Fixed-Variance Mixture of {G}aussians},
  author={Huix, Tom and Korba, Anna and Durmus, Alain and Moulines, Eric},
  journal={arXiv preprint arXiv:2406.04012},
  year={2024}
}

@article{yi2023bridging,
  title={Bridging the gap between variational inference and {W}asserstein gradient flows},
  author={Yi, Mingxuan and Liu, Song},
  journal={arXiv preprint arXiv:2310.20090},
  year={2023}
}

@article{YCY,
  author  = {Rentian Yao and Xiaohui Chen and Yun Yang},
  title   = {Wasserstein Proximal Coordinate Gradient Algorithms},
  journal = {J. Mach. Learn. Res},
  year    = {2024},
  volume  = {25},
  number  = {269},
  pages   = {1--66}
}

@inproceedings{balasubramanian2022towards,
  title={Towards a theory of non-log-concave sampling: first-order stationarity guarantees for {Langevin Monte Carlo}},
  author={Balasubramanian, Krishna and Chewi, Sinho and Erdogdu, Murat A and Salim, Adil and Zhang, Shunshi},
  booktitle={Conference on Learning Theory},
  pages={2896--2923},
  year={2022},
  organization={PMLR}
}

@article{lu2019accelerating,
  title={Accelerating {Langevin} sampling with birth-death},
  author={Lu, Yulong and Lu, Jianfeng and Nolen, James},
  journal={arXiv preprint arXiv:1905.09863},
  year={2019}
}

@article{chen2023accelerating,
  title={Accelerating optimization over the space of probability measures},
  author={Chen, Shi and Li, Qin and Tse, Oliver and Wright, Stephen J},
  journal={arXiv preprint arXiv:2310.04006},
  year={2023}
}

@article{wang2022accelerated,
  title={Accelerated information gradient flow},
  author={Wang, Yifei and Li, Wuchen},
  journal={J. Sci. Comput.},
  volume={90},
  pages={1--47},
  year={2022},
  publisher={Springer}
}

@article{wang2020information,
  title={Information {Newton's} flow: second-order optimization method in probability space},
  author={Wang, Yifei and Li, Wuchen},
  journal={arXiv preprint arXiv:2001.04341},
  year={2020}
}

@article{cheng2024particle,
  title={Particle-based variational inference with generalized {W}asserstein gradient flow},
  author={Cheng, Ziheng and Zhang, Shiyue and Yu, Longlin and Zhang, Cheng},
  journal={Advances in Neural Information Processing Systems},
  volume={36},
  year={2024}
}

@article{saeedi2017variational,
  title={Variational particle approximations},
  author={Saeedi, Ardavan and Kulkarni, Tejas D and Mansinghka, Vikash K and Gershman, Samuel J},
  journal={J. Mach. Learn. Res},
  volume={18},
  number={69},
  pages={1--29},
  year={2017}
}

@inproceedings{liu2019understanding,
  title={Understanding and accelerating particle-based variational inference},
  author={Liu, Chang and Zhuo, Jingwei and Cheng, Pengyu and Zhang, Ruiyi and Zhu, Jun},
  booktitle={International Conference on Machine Learning},
  pages={4082--4092},
  year={2019},
  organization={PMLR}
}

@article{durmus2017nonasymptotic,
  title={NONASYMPTOTIC CONVERGENCE ANALYSIS FOR THE UNADJUSTED {Langevin} ALGORITHM},
  author={Durmus, Alain and Moulines, {\'E}ric},
  journal={Ann. Appl. Probab.},
  volume={27},
  number={3},
  pages={1551--1587},
  year={2017}
}
}

\end{document}